\documentclass[
reqno]{amsart}
\usepackage{latexsym,amsmath,amssymb,amscd,amsthm}
\usepackage[all]{xy}
\usepackage{enumerate}
\usepackage{xcolor}
\usepackage{hyperref}

\def\today{\ifcase \month \or
   January \or February \or March \or April \or
   May \or June \or July \or August \or
   September \or October \or November \or December \fi
   \space\number\day , \number\year}

\makeatletter
  \newcommand\@dotsep{4.5}
  \def\@tocline#1#2#3#4#5#6#7{\relax
     \ifnum #1>\c@tocdepth 
     \else
     \par \addpenalty\@secpenalty\addvspace{#2}%
     \begingroup \hyphenpenalty\@M
     \@ifempty{#4}{%
     \@tempdima\csname r@tocindent\number#1\endcsname\relax
        }{%
         \@tempdima#4\relax
           }%
      \parindent\z@ \leftskip#3\relax \advance\leftskip\@tempdima\relax
      \rightskip\@pnumwidth plus1em \parfillskip-\@pnumwidth
       #5\leavevmode\hskip-\@tempdima #6\relax
       \leaders\hbox{$\m@th
       \mkern \@dotsep mu\hbox{.}\mkern \@dotsep mu$}\hfill
       \hbox to\@pnumwidth{\@tocpagenum{#7}}\par
       \nobreak
        \endgroup
         \fi}
\makeatother

\DeclareMathOperator\lie{L}

\begin{document}
	
\definecolor{forest}{RGB}{0,100,0}

\makeatletter
\@addtoreset{figure}{section}
\def\thefigure{\thesection.\@arabic\c@figure}
\def\fps@figure{h,t}
\@addtoreset{table}{bsection}

\def\thetable{\thesection.\@arabic\c@table}
\def\fps@table{h, t}
\@addtoreset{equation}{section}
\def\theequation{
\arabic{equation}}
\makeatother

\newcommand{\bfi}{\bfseries\itshape}
\newtheorem{theorem}{Theorem}
\newtheorem{corollary}[theorem]{Corollary}
\newtheorem{criterion}[theorem]{Criterion}
\newtheorem{lemma}[theorem]{Lemma}
\newtheorem{proposition}[theorem]{Proposition}

\theoremstyle{definition}
\newtheorem{definition}[theorem]{Definition}
\newtheorem{example}[theorem]{Example}
\newtheorem{notation}[theorem]{Notation}
\newtheorem{problem}[theorem]{Problem}
\newtheorem{remark}[theorem]{Remark}
\numberwithin{theorem}{section}
\numberwithin{equation}{section}

\newcommand{\todo}[1]{\vspace{5 mm}\par \noindent
\framebox{\begin{minipage}[c]{0.85 \textwidth}
\tt #1 \end{minipage}}\vspace{5 mm}\par}

\renewcommand{\1}{{\bf 1}}

\newcommand{\hotimes}{\widehat\otimes}

\newcommand{\Ad}{{\rm Ad}}
\newcommand{\Alt}{{\rm Alt}\,}
\newcommand{\Ci}{{\mathcal C}^\infty}
\newcommand{\comp}{\circ}
\newcommand{\wt}{\widetilde}

\newcommand{\ad}{\textrm{ad}}
\newcommand{\bp}{{\bf p}}
\newcommand{\bs}{{\bf s}}
\newcommand{\bt}{{\bf t}}
\newcommand{\ph}{\text{\bf P}}
\newcommand{\conv}{{\rm conv}}
\newcommand{\de}{{\rm d}}
\newcommand{\ee}{{\rm e}}
\newcommand{\ev}{{\rm ev}}
\newcommand{\fimes}{\mathop{\times}\limits}
\newcommand{\id}{{\rm id}}
\newcommand{\ie}{{\rm i}}
\newcommand{\End}{{\rm End}\,}
\newcommand{\Gr}{{\rm Gr}}
\newcommand{\GL}{{\rm GL}}
\newcommand{\Hilb}{{\bf Hilb}\,}
\newcommand{\Hom}{{\rm Hom}}
\renewcommand{\Im}{{\rm Im}}
\newcommand{\Ker}{{\rm Ker}\,}
\newcommand{\lf}{{\rm l}}
\newcommand{\Loc}{{\rm Loc}\,}
\newcommand{\nor}{{\rm nor}}
\newcommand{\pr}{{\rm pr}}
\newcommand{\Ran}{{\rm Ran}\,}
\newcommand{\rank}{{\rm rank}\,}
\newcommand{\sa}{{\rm sa}}
\newcommand{\spec}{{\rm spec}\,}
\newcommand{\supp}{{\rm supp}\,}
\newcommand{\tto}{\rightrightarrows}

\newcommand{\Tr}{{\rm Tr}\,}
\newcommand{\Tran}{\textbf{Trans}}

\newcommand{\CC}{{\mathbb C}}
\newcommand{\RR}{{\mathbb R}}
\newcommand{\NN}{{\mathbb N}}
\newcommand{\TT}{{\mathbb T}}

\newcommand{\G}{{\rm G}}
\newcommand{\U}{{\rm U}}
\newcommand{\Gl}{{\rm GL}}
\newcommand{\SL}{{\rm SL}}
\newcommand{\SU}{{\rm SU}}
\newcommand{\VB}{{\rm VB}}

\newcommand{\Ac}{{\mathcal A}}
\newcommand{\Bc}{{\mathcal B}}
\newcommand{\Cc}{{\mathcal C}}
\newcommand{\Dc}{{\mathcal D}}
\newcommand{\Ec}{{\mathcal E}}
\newcommand{\Fc}{{\mathcal F}}
\newcommand{\Gc}{{\mathcal G}}
\newcommand{\Hc}{{\mathcal H}}
\newcommand{\Ic}{{\mathcal I}}
\newcommand{\Jc}{{\mathcal J}}
\newcommand{\Kc}{{\mathcal K}}
\renewcommand{\Mc}{{\mathcal M}}
\newcommand{\Nc}{{\mathcal N}}
\newcommand{\Oc}{{\mathcal O}}
\newcommand{\Pc}{{\mathcal P}}
\newcommand{\Qc}{{\mathcal Q}}
\newcommand{\Rc}{{\mathcal R}}
\newcommand{\Sc}{{\mathcal S}}
\newcommand{\Tc}{{\mathcal T}}
\newcommand{\Uc}{{\mathcal U}}
\newcommand{\Vc}{{\mathcal V}}
\newcommand{\Wc}{{\mathcal W}}
\newcommand{\Xc}{{\mathcal X}}
\newcommand{\Yc}{{\mathcal Y}}
\newcommand{\Zc}{{\mathcal Z}}
\newcommand{\Ag}{{\mathfrak A}}
\renewcommand{\gg}{{\mathfrak g}}
\newcommand{\hg}{{\mathfrak h}}
\newcommand{\mg}{{\mathfrak m}}
\newcommand{\nng}{{\mathfrak n}}
\newcommand{\pg}{{\mathfrak p}}
\newcommand{\ug}{{\mathfrak u}}
\newcommand{\Gg}{{\mathfrak g}}
\newcommand{\Lg}{{\mathfrak L}}
\newcommand{\M}{\mathfrak{M}}
\newcommand{\Sg}{{\mathfrak S}}
\newcommand{\Ug}{{\mathfrak U}}

\markboth{}{}

\makeatletter
\title[Unitary group orbits versus groupoid orbits of normal operators]{Unitary group orbits versus groupoid orbits\\ of normal operators}
\author{Daniel Belti\c t\u a}
\address{Institute of Mathematics ``Simion Stoilow'' of the Romanian Academy,
P.O. Box 1-764, Bucharest, Romania}
\email{Daniel.Beltita@imar.ro} 
\author{Gabriel Larotonda}
\address{Departamento de Matemat\'ica, FCEYN-UBA, and Instituto Argentino de Matem\'atica, CONICET, Buenos Aires, Argentina}
\email{glaroton@dm.uba.ar}
\date{\today. 
\textbf{File name}: \texttt{BLv17.tex}}

\makeatother

\begin{abstract}
We study the unitary 
orbit of a normal operator $a\in \Bc(\Hc)$, regarded as a homogeneous space for the action of unitary groups associated with symmetrically normed ideals of compact operators. 
We show with an unified treatment that the orbit is a submanifold of the  
differing ambient spaces if and only if the spectrum of $a$ is finite, and in that case it is a closed submanifold. 
For arithmetically mean closed ideals, we show that nevertheless the orbit always has a natural manifold structure, modeled by the kernel of a suitable conditional expectation. 
When the spectrum of $a$ is not finite, we describe the closure of the orbits of $a$ for the different norm topologies involved. 
We relate these results to 
the action of the groupoid of the partial isometries 
via the moment map given by the range projection of normal operators. 
We show that all these groupoid orbits also have differentiable structures for which the target map is a smooth submersion. 
For any normal operator $a$ we also describe the norm closure of its groupoid orbit $\Oc_a$, which leads to necessary and sufficient spectral conditions on $a$ ensuring that $\Oc_a$ is norm closed and that $\Oc_a$ is a closed embedded submanifold of~$\Bc(\Hc)$. 
\end{abstract}

\maketitle

\tableofcontents

\section{Introduction}
\label{Sect1}

The literature concerning topological and geometrical aspects of unitary orbits $\Oc(a)$ of a normal (in particular: self-adjoint) operator $a$ acting in a separable Hilbert space~$\Hc$ is vast, but disperse, going back to the seminal work of Fialkow \cite{Fi75}, and related results before that. 
In 
this paper 
we 
bring together central aspects of this recurring theme in operator ideals 
and operator algebras, 
in a unified presentation, obtaining also generalizations of those results. 
Specifically, we 
study 
the differentiable structure of the norm closure of the unitary orbits. 
The natural framework for studying these unitary orbit closures turns out to be provided by the Banach-Lie groupoid of partial isometries introduced in \cite{OS}.

In more detail, assume that we act on a normal operator $a\in \Bc(\Hc)$ via 
$\pi(u):=u\cdot a:=uau^*$, with the Banach-Lie group $\U_{\Jc}$ of unitary operators in $\1+\Jc$ for a symmetrically normed ideal $\Jc\subseteq\Bc(\Hc)$. 
The differential of this map at the identity operator $\1\in\U_{\Jc}$ is $\pi_{*\1}=-\ad\, a$, $x\mapsto[x,a]=xa-ax$, and the stabilizer subgroup of this action is denoted $\U_{\Jc}(a)\subseteq \U_{\Jc}$. 
The orbit itself  $\Oc_{\Jc}(a):=\pi(\U_{\Jc})$ is a subset of the linear space  $\Bc(\Hc)$, but when $a\in \Ic$ 
for an operator ideal $\Ic$, 
then $\Oc_{\Jc}(a)\subseteq \Ic$ as well. 
There are 
several basic problems that are interlaced, and consist of finding necessary and sufficient conditions on the normal operator $a$ and the symmetrically normed ideals $\Jc$ and $\Ic$ so that the following assertions hold true: 
\begin{enumerate}
\item 
\label{Sect1_item1}
$\U_{\Jc}(a)
$ is a Lie subgroup of the Banach-Lie group $\U_\Jc$ with its subspace topology.
\item 
\label{Sect1_item2}
$\Oc_{\Jc}(a)\simeq \U_{\Jc}/\U_{\Jc}(a)$ is an (abstract) differentiable Banach manifold.
\item 
\label{Sect1_item3}
the action map $\pi\colon\U_{\Jc}\to\Oc_{\Jc}(a)$ has local continuous cross-sections.
\item 
\label{Sect1_item4}
the range of $\ad\,a:T_\1(\U_{\Jc})\to \Ic$ is closed.
\item 
\label{Sect1_item5}
the subset $\Oc_{\Jc}(a)\subseteq \Ic$ is closed.
\item 
\label{Sect1_item6}
the subset $\Oc_{\Jc}(a)\subseteq \Ic$ is an immersed submanifold
\item 
\label{Sect1_item7}
the subset $\Oc_{\Jc}(a)\subseteq \Ic$ is an embedded submanifold (with its inherited topology).
\end{enumerate}
We will show via an unified approach that, for any normal operator $a\in\Bc(\Hc)$, the first two of the above conditions can be obtained when 
the ideal $\Jc$ is arithmetically mean closed 
(cf. \cite{kw11})
or when $\Jc$ admits a dual pairing as in~\eqref{duality} below. 
Then we will show that the last four conditions are equivalent to each other, and in fact equivalent to the fact that $a$ has \textit{finite} spectrum 
(Theorem~\ref{locl}). 
There are several topologies involved (the operator norm topology, and the ones given by the norms of the ideals $\Ic$ and $\Jc$), and we deal with all of them systematically. 
This is done in Section \ref{sec:orbit}.

This brings out the problem of describing the closure of the orbit, when the spectrum of the normal operator $a$ is infinite. 
This is done in Section \ref{sec:closure}. 
If $\Fc(\Hc)$ is the ideal of finite-rank operators and $a\in\Ic$, then 
it turns out that 
$$
\Oc_{\Fc(\Hc)}(a)\subseteq 
\Oc_{\Ic}(a)\subseteq 
\{vav^*\mid v^*v=p_a\}=\overline{\Oc_{\Bc(\Hc)}(a)}^{\Vert\cdot\Vert}
\subseteq\Ic,
$$ 
where $p_a$ is the range projection of $a$. 
We prove that when $\Ic$ is a separable symmetrically normed ideal, the closures are equal
$$\overline{\Oc_{\Fc(\Hc)}(a)}^{\Vert\cdot\Vert_\Ic}
=\overline{\Oc_{\Ic}(a)}^{\Vert\cdot\Vert_\Ic}
=
\overline{\Oc_{\Bc(\Hc)}(a)}^{\Vert\cdot\Vert}.$$
Conversely, we show that if $\Oc_{\Bc(\Hc)}(a)\subseteq \overline{\Oc_{\Fc(\Hc)}(a)}^{\Vert\cdot\Vert_\Ic}$ whenever $0\le a\in \Ic$, then the symmetrically normed ideal~$\Ic$ is separable (Theorem~\ref{closures_th}).

In Section \ref{sec:groupoid} we find the natural framework for the above 
occurrence of the set of partial isometries with fixed initial space $\Vc_a=\{v\in \Bc(\Hc)\mid v^*v=p_a\}$. 
This set $\Vc_a\subseteq \Bc(\Hc)$ was studied by several authors and in particular Andruchow et. al in \cite{AnCoMb05} established  that $\Vc_a$ is a submanifold of $\Bc(\Hc)$.  

Denoting by $\Pc(\Hc)\subseteq\Vc(\Hc)$ the sets of orthogonal projections and partial isometries on $\Hc$, there are the groupoid $\Vc(\Hc)\tto\Pc(\Hc)$ (cf. \cite{OS}) and the natural map 
on the set of normal operators 
$\Bc(\Hc)^{\nor}\to\Pc(\Hc)$, $a\mapsto p_a$. 
These structures allow one to define a groupoid action, 
which defines an action groupoid  $\Vc(\Hc)\ast \Bc(\Hc)^{\nor}\tto \Bc(\Hc)^{\nor}$. 
For any $a\in\Bc(\Hc)^{\nor}$, its corresponding groupoid orbit is 
$$
\Oc_a=\{vav^*\mid (v,a)\in  \Vc(\Hc)\ast \Bc(\Hc)^{\nor}\}
=\{vav^*\mid  v^*v=p_a\}
\subseteq \Bc(\Hc)^{\nor}
$$
which for compact $a$ is exactly the uniform closure of the unitary orbit of $a$, as mentioned above. 
We study this orbit with $\Bc(\Hc)$ replaced by a general von Neumann algebra $\M\subseteq\Bc(\Hc)$, establishing fundamental differentiability properties which in the case of the orbit of a normal operator $a\in \Bc(\Hc)$ can be subsumed to the following facts: 
The mapping $\Vc_a\to \Oc_a$, $v\mapsto vav^*$, is a principal bundle whose structural group 
$$
K_a:=\{w\in \Vc(\Hc)\mid w^*w=p_a,\, waw^*=a\}\subseteq \Vc_a\cap \U(p_a\Hc)
$$
is a Banach-Lie group and a submanifold of $\Vc_a$; the inclusion map $\Oc_a\subseteq \Kc(\Hc)$ is smooth and its tangent map at every point is injective; and finally, the inclusion map $\Oc_{\Bc(\Hc)}(a)\hookrightarrow \Oc_a$  is smooth, and its tangent map at every point is injective 
(Corollary~\ref{smcl0}). 
Thus, in the special case of a compact operator $a$, this gives a good picture of the relation between the unitary orbit of $a$ and its closure, when both are regarded as smooth manifolds 
(Corollary~\ref{smcl}). 
Moreover, for any normal operator $a\in\Bc(\Hc)^{\nor}$ we describe the norm closure of its groupoid orbit $\overline{\Oc_a}$ 
(Proposition~\ref{P3}), which leads to necessary and sufficient spectral conditions on $a$ ensuring that $\Oc_a$ is norm closed (Corollary~\ref{C4}) and that $\Oc_a$ is a closed embedded submanifold of~$\Bc(\Hc)$ 
(Theorems \ref{immersed}  and \ref{top_th}).

\section{Basic definitions and background}
\label{Sect2}

\begin{notation}
We denote by $\Hc$ a separable infinite-dimensional complex Hilbert space, by $\Bc(\Hc)$ the set of all bounded linear operators on $\Hc$. 
The spectrum of any operator $a\in\Bc(\Hc)$ is denoted by $\spec(a)$. 

For any subset $S\subseteq \Bc(\Hc)$ we denote
$$
S^{\sa}
:=\{a\in S\mid a=a^*\} 
\text{ and }
S^{\nor}
:=\{a\in S\mid aa^*=a^*a\}. 
$$ 
In particular $\Bc(\Hc)^{\sa}$ are the self-adjoint operators 
and $\ie \Bc(\Hc)^{\sa}$ are the skew-adjoint operators. We also let
$$
\U(\Hc):=\{u\in\Bc(\Hc)\mid uu^*=u^*u=\1\}
$$
where $\1\in\Bc(\Hc)$ is the identity operator. We denote by $\Fc(\Hc)$ the set of all finite-rank operators on $\Hc$ and by $\Kc(\Hc)$ the set of all compact operators on~$\Hc$. We denote by $\Sg_p(\Hc)$ the $p$-th Schatten ideal if $1\le p\le\infty$, where $\Sg_\infty(\Hc):=\Kc(\Hc)$. 

Since the Hilbert space $\Hc$ is fixed, we may drop it from the notation sometimes, for the sake of simplicity. 
\end{notation}

\begin{definition}\label{adm}
A \emph{symmetrically normed ideal} is a two-sided ideal $\Jc\subseteq \Bc(\Hc)$ endowed with a norm $\Vert\cdot\Vert_{\Jc}$ satisfying 
$$
\Vert axb\Vert_{\Jc}\le\Vert a\Vert \Vert x\Vert_{\Jc} \Vert b\Vert
$$
for all $x\in\Jc$ and $a,b\in\Bc(\Hc)$, where $\Vert\cdot\Vert$ denotes the usual operator norm of $\Bc(\Hc)$. 
We also assume \textit{that $\Jc$ with its norm is a Banach space}. 
We choose the normalization $\Vert x\Vert _{\Jc}=\Vert x\Vert $ for any 
rank-one projection $x=\xi\otimes\overline{\xi}$, $\xi\in\Hc$. 
(Recall from \cite[Ch. III, Th. 1.1]{GoKr69} that any nonzero two-sided ideal of $\Bc(\Hc)$ contains all finite-rank operators.)
\end{definition}

We now recall some fundamental facts on symmetrically normed ideals. 
Our main reference on this subject is Gohberg and Krein's book \cite{GoKr69}. If $x=u\vert x\vert $ is the polar decomposition of $x\in \Jc$, then
$$
\Vert x\Vert_{\Jc}=\Vert x^*\Vert _{\Jc}
=\Vert\vert x\vert\Vert_{\Jc}\quad\textrm{ and }\quad \Vert x\Vert _{\Jc}\le \Vert x\Vert .
$$
If $0\le a\le b$ then by Douglas' Lemma \cite{douglas} there exist a contraction $0\le c\le \1$ such that $a=bc$, hence 
$$
0\le a\le b\quad \Rightarrow \quad \Vert a\Vert _{\Jc}\le \Vert b\Vert _{\Jc}.
$$

Let $x=u\vert x\vert $ be the polar decomposition of a compact operator, and let $\vert x\vert=
\lim\limits_{n\to\infty} x_n=\lim\limits_{n\to\infty} \sum\limits_{k=1}^n s_k  p_k$ be the Schmidt decomposition of 
$\vert x\vert$. 
Here $s_k(x)\ge 0$ are the \textit{singular values} of $x$
 (that is, the eigenvalues of $\vert x\vert $), and $p_k=e_k\otimes \overline{e_k}$ are the rank-one ortho-projections of $\vert x\vert $. 
 The convergence is in the operator norm since $s_k(x)\downarrow 0$. 
 If the symmetric norm is given, we can consider the operators $x\in \Kc(\Hc)$ with $\sup\limits_{n\ge 1} \Vert x_n\Vert _{\Jc}<\infty$, 
 and then define $\Vert x\Vert _{\Jc}'=\lim\limits_{n\to\infty} \Vert x_n\Vert _{\Jc}$. 
 \textit{The norms $\Vert\cdot\Vert_{\Jc}$ and $\Vert\cdot\Vert_{\Jc}'$ agree on the ideal of finite rank operators}. 

\begin{notation}
The set of all operators $x\in\Kc(\Hc)$ with $\sup_n \Vert x_n\Vert _{\Jc}<\infty$ is the \textit{maximal ideal} $\Jc^M$ for the given norm $\Vert \cdot\Vert _{\Jc}$, and it is always a Banach space. 
We also consider the  \textit{minimal ideal} $\Jc^m$ which is the closure in the norm $\Vert \cdot\Vert_{\Jc}$ of the finite rank operators, hence it is also complete. 
The minimal ideal is then a separable Banach space, and it is not hard to check that the ideal $\Jc$ is separable if and only if it coincides with its minimal ideal $\Jc^m$ \cite[Ch. III, Th. 6.2]{GoKr69}.
\end{notation}

Since $0\le x_n=\sum\limits_{k=1}^n s_k(x)p_k\le \vert x\vert $, and the norm only depends on the positive part of $x=u\vert x\vert $, we have $\Vert x_n\Vert _{\Jc}\le \Vert x\Vert _{\Jc}<\infty$ if $x\in \Jc$, then $\Vert x\Vert _{\Jc}'\le \Vert x\Vert _{\Jc}<\infty$ thus
$$
\Jc^m\subseteq \Jc \subseteq \Jc^M.
$$
Moreover, for any finite rank operator $x$ we have $\vert x\vert =\sum\limits_{k=1}^n s_k(x)p_k$ therefore $\Vert x\Vert _{\Jc}\le \sum\limits_{k=1}^n s_k(x)$, hence
$$
\Sg_1(\Hc)\subseteq  \Jc^M\subseteq \Kc(\Hc) \textrm{ and } \Vert x\Vert \le \Vert x\Vert _{\Jc}'\le \Vert x\Vert _1
$$
for all $x\in \Kc$.

\begin{definition}[arithmetically mean closed ideals] 
The ideal $\Jc$ is \textit{arithmetically mean closed}  if $y\in \Jc$ implies $x\in \Jc$ 
whenever $x,y\in \Kc(\Hc)$ and $\sum\limits_{k=1}^n s_k(x)\le \sum\limits_{k=1}^n s_k(y)$ for all $n\in\NN$. 
We will refer to this property as \textit{am-closed} for short. 
The maximal and minimal ideals for a given symmetric norm $\Vert \cdot\Vert _{\Jc}$ are always am-closed by the dominance property and $\Vert x\Vert _{\Jc}'\le \Vert y\Vert _{\Jc}'$ for those ideals. 
(See the dominance property in \cite[Chapter III, \S 4]{GoKr69}.)
\end{definition}

\begin{remark}[ideals with two norms]
It would be nice to characterize the symmetrically normed ideals $\Jc$ such that $\Jc$ is complete with the maximal norm $\Vert \cdot\Vert _{\Jc}'$. Equivalently, since the maximal ideal is complete, $\Jc$ must be closed in the maximal ideal. 
By the open mapping theorem, and since $\Vert \cdot\Vert _{\Jc}'\le \Vert \cdot\Vert _{\Jc}$ always holds, the ideal $\Jc$ is closed in $\Jc^M$ if and only if there exists a constant $C>0$ such that $\Vert \cdot\Vert _{\Jc}\le C\Vert \cdot\Vert _{\Jc}'$; that is, if the norm of $\Jc$ is equivalent to the maximal norm. Varga \cite[Theorem 2 and Remark 5]{varga} supplied (a family of) examples of complete normed ideals that are not am-closed; but his family of examples are closed in the maximal norm. 
\end{remark}

\begin{definition}[conditional expectations]
\label{discrete}
For any subset $A\subseteq\NN$ 
and any family of mutually orthogonal projections  $\mathcal P=\{p_k\}_{k\in A}$ with $\sum\limits_{k\in A}p_k=\1$ (with convergence in the strong operator topology if $A$ is infinite) let  $E:\Bc(\Hc)\to \Bc(\Hc)$ denote the corresponding conditional expectation 
\begin{equation}\label{discrete_eq1}
E(x)=\sum_{k\in A} p_k xp_k
\end{equation}
with convergence in the strong operator topology for any $x\in\Bc(\Hc)$. 
The range of~$E$ is 
\begin{equation}\label{discrete_eq2}
\Ran E=\{p_k\mid k\in A\}'.
\end{equation}
We have $E(x)\in \Kc(\Hc)$ if $x\in \Kc(\Hc)$. 
Moreover, for arbitrary $x\in\Bc(\Hc)$, 
\begin{equation}\label{singexp}
\sum_{k=1}^n s_k(E(x))\le \sum_{k=1}^n s_k(x)
\end{equation}
for all $n\in\mathbb N$ 
by \cite[Ch. II, Th. 5.1]{GoKr69}. 
In particular, for $n=1$, we obtain $\Vert E\Vert=1$. 
\end{definition}

In what follows we will examine in more detail what happens when we restrict and co-restrict $E$ to operator ideals.

\begin{remark}[continuity of $E\vert_{\Jc}$]
\label{contexp}
In the setting of Definition~\ref{discrete}, if $\Jc$ is a symmetrically normed ideal and $x\in \Jc$, from \eqref{singexp} we obtain $E(x)\in  \Jc^M$ and moreover by the dominance property
\begin{equation}\label{contexp_eq1}
 \Vert E(x)\Vert _{\Jc}'\le \Vert x\Vert _{\Jc}'\le \Vert x\Vert _{\Jc},
\end{equation}
which implies that the operator $E\vert_{\Jc}\colon(\Jc,\Vert\cdot\Vert_\Jc')\to(\Jc^M,\Vert\cdot\Vert_\Jc')$ satisfies $\Vert E\vert_{\Jc}\Vert=1$. 
Here we may have $E(x)\in \Jc^M\setminus\Jc$ for some $x\in\Jc$. 
In fact, if all the projections $p_k$ are rank one, 
the condition $E(\Jc)\subseteq \Jc$ is equivalent to the fact that $\Jc$ is am-closed by \cite[Theorem 4.5]{kw11}. 
If $E(\Jc)\subseteq \Jc$ then, by the closed graph theorem, the operator  $E\vert_{\Jc}\colon(\Jc,\Vert\cdot\Vert_\Jc)\to(\Jc,\Vert\cdot\Vert_\Jc)$ is also continuous 
although its norm need not be equal to~1. 
However, if $\Jc$ is separable then it is equal to $\Jc^m$ thus it is arithmetically closed and both norms agree. 

We note also that if the collection of projections is \textit{finite}, then $E(\Jc)\subseteq \Jc$ and the continuity of $E\vert _{\Jc}$ is obvious.
\end{remark}

\begin{remark}[conditional expectation of a normal diagonalizable operator]\label{conde}
An operator $a\in \Bc(\Hc)^{\nor}$ is called \emph{diagonalizable} if $\Hc$ is spanned by the eigenvectors of~$a$.   
If this is the case, let $A\subseteq\NN$ for which $\spec(a)$ can be labeled as $\{\lambda_k\mid k\in A\}\subseteq\CC$, with $\lambda_j\ne\lambda_k$ if $j\ne k$. 
For any $k\in A$ let $p_k\in\Bc(\Hc)$ denote the orthogonal projection onto the eigenspace $\Ker(a-\lambda_k\1)$, hence the hypothesis that $a$~is a normal operator ensures that $p_jp_k=0$ if $j\ne k$. 
The hypothesis that the operator~$a$ is diagonalizable then implies $\sum\limits_{k\in A}p_k=\1$, and  
we can then write 
\begin{equation}\label{alambdda}
a=\sum_{k\in \mathbb A} \lambda_k p_k
\end{equation}
with convergence in the strong operator topology. 
This implies $\{p_k\mid k\in A\}'=\{a\}'$ and, 
if $E\colon\Bc(\Hc)\to\Bc(\Hc)$ is the conditional expectation associated 
to the family of mutually orthogonal projections $\{p_k\mid k\in A\}$ as in Definition~\ref{discrete}, then 
we have 
\begin{equation}\label{conde_eq1}
\Ran E=\{a\}'
\end{equation} by \eqref{discrete_eq2}. 
Moreover, $E(\Kc(\Hc))=\{a\}'\cap\Kc(\Hc)$.

If $\Ic$ is any symmetrically normed ideal and $a\in\Bc(\Hc)^\nor$ is diagonalizable as above, 
then we obtain the continuous map 
\begin{equation*}
E\vert_{\Ic}\colon \Ic\to \{a\}'\cap \Ic^M
\end{equation*} 
by Remark~\ref{contexp}. 
If moreover either $\spec(a)$ is finite and $\Ic$ is any symmetrically normed ideal, 
or the symmetrically normed $\Ic\subsetneqq\Bc(\Hc)$ is am-closed, 
then, by Remark~\ref{contexp} again, we have the continuous map 
\begin{equation*}
E\vert_{\Ic}\colon \Ic\to \{a\}'\cap \Ic. 
\end{equation*} 
 However, if $\Ic\subsetneqq\Kc(\Hc)$, we have 
 $E(\Kc(\Hc))\not\subseteq \{a\}'\cap \Ic$  
 since $E(\Kc(\Hc)\cap\{a\}')=\Kc(\Hc)\cap\{a\}'\not\subseteq \Ic\cap\{a\}'$ 
 irrespective of wheather or not  $a\in\Ic$. 
\end{remark}

\begin{definition}[congruence unitary groups] 
For a given operator ideal $\Jc\subseteq \Bc(\Hc)$, we define its corresponding unitary group 
$$
\U_{\Jc}=\U_{\Jc}(\Hc):=(\1+\Jc)\cap\U(\Hc). 
$$
This is a Banach-Lie group if $\Jc$ is a symmetrically normed ideal, 
and then its Lie algebra  is 
$$
\ug_{\Jc}=\ug_{\Jc}(\Hc):=\{x\in\Jc\mid x^*=-x\}=\ie\Jc^{\sa}.
$$
In fact, $\U_{\Jc}$ is an algebraic subgroup of the group of invertible elements of the Banach algebra  $\Jc\oplus\CC\1$, therefore it is a Banach-Lie subgroup with the inherited manifold topology of the group of invertible operators, the later being an open subset of the Banach algebra. 
See \cite{HK77} 
for more details. 
\end{definition}

\begin{remark}\label{uconnected}
The group $\U_{\Jc}$ is always connected. 
To see this, note that each $u\in\U_{\Jc}$ has a Borel logarithm $z^*=-z\in \Bc(\Hc)$ such that $\Vert z\Vert \le \pi$.
 Consider the holomorphic map $F:\mathbb C\to \mathbb C$ given by $F(\lambda)=\lambda^{-1}(e^{\lambda}-1)$, and note that the zeroes of $F$ are located at $2k\pi \ie$, $k\in \mathbb Z\setminus\{0\}$. 
Since $\spec(z)\subset [-\pi,\pi]$, we have $0\notin \spec(F(z))=F(\spec(z))$ thus $F(z)$ is invertible in $\Bc(\Hc)$. Since $u=1+x$ for some $x\in \Jc$,  $x+1=u=e^z$  thus $e^z-1=x$. But then 
$$
z=F(z)^{-1}F(z)z=F(z)^{-1}(e^z-1)=F(z)^{-1}x\in \Jc.
$$
From this, we can now see that $t\mapsto e^{tz}$ is a continuous path in $\Uc_{\Jc}$ joining the operators $\1$ and $u$.
\end{remark}

\section{Unitary orbits of normal operators}
\label{sec:orbit}

In this section we establish that, for any symmetrically normed ideals $\Ic,\Jc\subseteq\Bc(\Hc)$, the conditions \eqref{Sect1_item3}--\eqref{Sect1_item7} 
in Section~\ref{Sect1} are equivalent to the fact that the normal operator~$a$ has finite rank (Theorem~\ref{locl}). 

We begin by considering the adjoint action of unitary groups in the space of bounded linear operators: 

\begin{definition}[adjoint action and orbits]
The \emph{adjoint action} of $\U_{\Jc}$ on $\Bc(\Hc)$ is given by the map
$$\Ad_{\Jc}\colon \U_{\Jc}\times \Bc(\Hc)\to\Bc(\Hc),\quad 
(u,a)\mapsto\Ad_{\Jc}(u)a:=uau^*$$
and for every $a\in\Bc(\Hc)$ its corresponding $\U_{\Jc}$-orbit is 
$$\Oc_{\Jc}(a):=\{uau^*\mid u\in \U_{\Jc}\}.$$ 
\end{definition}

\begin{remark}\label{incl}
Let $\Jc, \Ic\subseteq\Bc(\Hc)$ be 
any ideals with $a\in \Ic$.
We claim that 
	$$
	\Oc_{\Jc}(a)\subseteq a+ \Ic\cap \Jc\subseteq \Ic .
	$$
	In fact 
	$z:=uau^*\in\Ic$ if $a\in \Ic$ and $u\in \U_{\Jc}$. 
	On the other hand, writing $u=\1+x\in \U_{\Jc}$ with $x\in \Jc$, 
	we have 
	$$
	z
	=(\1+x)a(\1+x^*)=a+xa+ax^*+xax^*=a+j \in a+\Jc
	$$	
where $j:=xa+ax^*+xax^*\in\Jc$ since $\Jc$ is an ideal. 
Also $j=z-a\in \Ic-\Ic=\Ic$, hence $j\in\Ic\cap\Jc$, 
and then $uau^*=z=a+j\in a+\Ic\cap\Jc$, as claimed.
\end{remark}

\subsection{The action map and the smooth structure of the orbit}

\begin{definition}
\label{action}
For any symmetrically normed ideals $\Jc,\Ic\subseteq\Bc(\Hc)$ and $a\in\Ic$, we define
$$
\pi=\pi^a_{\Jc,\Ic}:\U_{\Jc}\to \Oc_{\Jc}(a),\quad u\mapsto uau^*.
$$	
\end{definition}

Writing $u=\1+x,u_0=\1+x_0\in \1+\U_{\Jc}$, we obtain
$$
\Vert uau^*-u_0au_0^*\Vert _{\Ic}=\Vert (u-u_0)a+a(u^*-u_0^*)\Vert _{\Ic}\le 2\Vert u-u_0\Vert  \Vert a\Vert _{\Ic}\le 2\Vert u-u_0\Vert _{\Jc} \Vert a\Vert _{\Ic}
$$
therefore $\pi^a_{\Jc,\Ic}$ is continuous. 
Moreover a similar argument shows that $\pi\colon\U_{\Jc}\to \Ic$ is smooth, with its differential 
$$\pi_{*u}\colon T_u \U_{\Jc}\to \Jc,\quad 
\pi_{*u}\,v=vau^*-uau^*vu^*=u[u^*v,a]u^*.
$$
Here we use the identification 
$$T_u \U_{\Jc}=\{v\in\Bc(\Hc)\mid u^*v\in\ug_{\Jc}=\ie{\Jc}^{\sa}\}$$ 
and then $u[u^*v,a]u^*\in \Jc^{\sa}$ is self-adjoint if $v\in T_u \U_{\Jc}$ 
and $a\in\Ic^{\sa}$.

Keeping $a\in\Ic$ fixed as above, we now define 
$$\delta_a\colon\Bc(\Hc)\to\Ic,\quad \delta_a (v):=-\ad(a)(v)=[v,a].
$$ 
Then $\delta_a\vert _{\ie\Jc^{\sa}}=\pi_{*\1}$ and by Remark \ref{incl} we obtain 
\begin{equation}\label{pi}
\pi_{*\1}=\delta_a\vert _{\ie\Jc^\sa}: \ie\Jc^{\sa}\to  \Ic\cap\Jc. 
\end{equation}
Note also that  $\pi_{*\1}$ (or equivalently, $\delta$) is continuous 
if we take it as a map from $\Jc$ into any of the three spaces $\Jc,\Ic,\Bc(\Hc)$.

In Lemma~\ref{splits} below we use the above notation and Remark~\ref{conde}.

\begin{lemma}\label{splits} 
Let $\Ic,\Jc$ be symmetrically normed ideals 
and $a\in\Bc(\Hc)^{\nor}$ be diagonalizable with its corresponding conditional expectation 
 $E\vert_\Ic:\Ic \to \{a\}'\cap \Ic^M$.  
 Then the following assertions hold. 
\begin{enumerate}[{\rm(i)}]
\item\label{splits_item1} $\Ker((\id_{\Bc(\Hc)}-E)\vert_\Ic)= \Ker (\delta_a\vert_\Ic)$.
\item\label{splits_item2} 
If $\delta_a(\Jc^{\sa})\subseteq \Ic$  and the operator $\delta_a\vert _{\ie\Jc^\sa}: \ie\Jc^{\sa}\to  \Ic$ has closed range, 
then $\spec(a)$ is finite. 
\end{enumerate}
\end{lemma}

\begin{proof}
\eqref{splits_item1}
We clearly have $\Ker\delta_a=\{a\}'$, while $\{a\}'=\Ran E$ by~\eqref{conde_eq1}. 
Moreover, since $E^2=E$, we have $\Ran E=\{x\in\Bc(\Hc)\mid E(x)=x\}=\Ker(\id_{\Bc(\Hc)}-E)$. 
Therefore $\Ker(\id_{\Bc(\Hc)}-E)= \Ker \delta_a$ and, intersecting both sides of this equality with~$\Ic$, we obtain the assertion.

\eqref{splits_item2}
The hypothesis directly implies that $\delta_a\colon \Jc^\sa\to\Ic$ has closed range. 
Moreover, since both $\delta_a\colon \Jc^\sa\to\Bc(\Hc)$  and 
the inclusion map $\Ic\hookrightarrow\Bc(\Hc)$ are continuous, 
while $\delta_a(\Jc)\subseteq\Ic$ by hypothesis, 
it follows by the closed graph theorem that the linear operator
$\delta_a\vert _{\Jc^\sa}\colon \Jc^\sa\to\Ic$ is continuous.  Now, by the open mapping theorem 
 there exists a constant $C>0$ such that whenever $y=\delta_a(z)$ with $z\in \Jc^\sa$,  there exists $z_0\in \Jc^\sa$ such that $\delta_a(z)=\delta_a(z_0)$ and $\Vert z_0\Vert _{\Jc}\le C\Vert \delta(z)\Vert _{\Ic}$. 
 
 Assuming that $\spec(a)$ is infinite,  
 we may take $A=\{1,2,\dots\}$ in Remark~\ref{conde}. 
 Selecting a suitable subsequence of $\spec(a)$, we may assume that  
 $\vert \lambda_j-\lambda_{j+1}\vert <1/(4C)$ for every $j\ge 1$. 
 Select 
 any $\xi_j\in \Ran p_j$ with $\Vert \xi_j\Vert=1$, and then  $a\xi_j=\lambda_j \xi_j$,  
 and consider the rank-two operator $z:=\xi_j\otimes \overline{\xi_{j+1}}+\xi_{j+1}\otimes \overline{\xi_j}\in\Fc(\Hc)^\sa\subseteq\Jc^\sa$.
 Then 
$$
 \vert z\vert =z^2= \xi_j\otimes \overline{\xi_j} + \xi_{j+1}\otimes\overline{\xi_{j+1}}
 \ge \xi_j\otimes \overline{\xi_j},
$$
thus $\Vert z\Vert _{\Jc}'\ge \Vert \xi_j\otimes \overline{\xi_j}\Vert _{\Jc}'=1$. 
Note that 
$$za-za= (\lambda_j-\lambda_{j+1})\xi_j\otimes \overline{\xi_{j+1}} - (\lambda_j-\lambda_{j+1})\xi_{j+1}\otimes \overline{\xi_j},
$$
thus
$$
\vert [z,a]\vert ^2= \vert \lambda_j-\lambda_{j+1}\vert ^2 (e_j\otimes \overline{e_j} + e_{j+1}\otimes \overline{e_{j+1}})
$$
and then
$$
\Vert\delta_a(z)\Vert_{\Ic}=\Vert za-za\Vert _{\Ic} =\Vert [z,a]\Vert _{\Ic} \le 2\vert \lambda_j-\lambda_{j+1}\vert <\frac{1}{2C}.
$$
Now note that $p_jzp_j=p_{j+1}zp_{j+1}=0$ thus $E(z)=0$. 
Using the remark at the beginning of this proof, we may select $z_0\in\Jc^\sa$ with $\delta_a(z_0)=\delta_a(z)$ and $\Vert z_0\Vert _{\Jc}\le C \Vert \delta_a(z)\Vert _{\Ic}$. 
Then $\delta_a(z-z_0)=0$,  
hence $(\id_\Hc-E)(z-z_0)=0$ by Assertion~\eqref{splits_item1}. 
Therefore 
$$
z=z-E(z)=z_0-E(z_0)
$$
and, using \eqref{contexp_eq1}, we then obtain $\Vert z\Vert _{\Jc}'\le 2\Vert z_0\Vert _{\Jc}'\le 2\Vert z_0\Vert _{\Jc}$. 
We have seen above that $\Vert z\Vert _{\Jc}'\ge 1$, hence
$$
1\le \Vert z\Vert _{\Jc}'\le 2\Vert z_0\Vert _{\Jc}\le 2C\Vert \delta_a(z)\Vert _{\Ic}<1,
$$
a contradiction that concludes the proof. 
\end{proof}

In connection with Lemma~\ref{splits}, we note that for the case of the maximal ideal of a given symmetric norm, the equivalence between finite rank and closed range of $\delta$ was proved before in \cite[Lemma 3.2]{ChiDIyL13}.

Now we extend the previous result to the orbit of \textit{any}  
normal bounded operator. 
We first study the case of self-adjoint operators in Theorem~\ref{noncompact} following ideas in \cite[Theorem 4.3]{AL10}, 
and then the general normal operators in Theorem~\ref{au_th}.
Regarding the differentiable structure of the isotropy group and the orbit of a non-compact operator, see Subsection~\ref{dualpairs}.

\begin{theorem}\label{noncompact}
	Let $\Jc$ be a symmetrically normed ideal. 
	For $a\in \Bc(\Hc)^{\sa}$, consider 
	the map $\pi_{*\1}:\ie\Jc^{\sa}\to \Jc\subseteq\Bc(\Hc)$ defined as before. 
	If the range of $\pi_{*\1}$ is closed in either $\Bc(\Hc)$ or $\Jc$, then $\spec(a)$ is finite. 
\end{theorem}

\begin{proof}
	For any $x\in \Bc(\Hc)$ we use the above notation $\delta_x=-[x,\cdot]$. 
	As noted before $\pi_{*\1}=\delta_a\vert _{\ie\Jc^{\sa}}$ or its restriction to (the skew-adjoint part) of $\Jc$ and $\Kc(\Hc)$, respectively. 
	Since $a=a^*$, hence both $\Bc(\Hc)^\sa$ and $\ie\Bc(\Hc)^\sa$ are invariant under $\ie\delta_a$, it is easily seen that the range of 
	$\delta_a\vert_{\Jc^\sa}$ is closed in $\Bc(\Hc)$ or $\Jc$ if and only if the range of $\delta_a\vert_\Jc$ is closed in $\Bc(\Hc)$ or $\Jc$, respectively.  
	
	Since $a\in \Bc(\Hc)^{\sa}$, 
	the Hilbert space $\Hc$ can be decomposed in two orthogonal subspaces $\Hc=\Hc_c\oplus\Hc_{p}$ which reduce $a$, 
	where $\Hc_{p}$ is the closed linear subspace spanned by the eigenvectors of~$a$. 
	Then the operator $a_c:=a\vert _{\Hc_c}\in \Bc(\Hc_c)$ has continuous spectrum (i.e., no eigenvalues), 
	while the spectrum of $a_{p}:=a\vert _{\Hc_{p}}$ has a dense subset of eigenvalues. 
	If the range of $\delta_a$ is closed in $\Bc(\Hc)$, we claim that both $\delta_{a_c}$ and $\delta_{a_{p}}$ have closed range in~$\Bc(\Hc)$. 
	Suppose $x_n\in\Bc(\Hc)$ is such that $\delta_{a_c}(x_n)\to y$, then $y_n=x_n\oplus 0 \in \Bc(\Hc)$ satisfies
	$$
	\delta_a(y_n)=\delta_{a_c}(x_n)\oplus 0\to y\oplus 0
	$$
	in $\Bc(\Hc)$, and thus $y\oplus 0=\delta_a(x)$. If one writes this equality in matrix form (in terms of the decomposition $\Hc=\Hc_c\oplus\Hc_{p}$), one has
	$$
	\left( \begin{array}{ll} y & 0 \\ 0 & 0 \end{array} \right) = 
	\left( \begin{array}{ll} x_{11}a_c-a_c x_{11} & x_{12}a_{p}-a_c x_{12} \\ x_{21}a_c-a_{p}x_{21} & x_{22}a_{p}-a_{p}x_{22} \end{array} \right),
	$$ 
	and therefore $y=\delta_{a_c}(x_{11})$. 
	Analogously one proves that the range of $\delta_{a_{p}}$ is closed.
	
	We now show  that $\Hc_c=\{0\}$: note that $\delta_{a_c}:\Kc(\Hc)\to \Kc(\Hc)$ has trivial kernel. 
	Indeed, if $x\ne 0$ is a compact operator commuting with $a_c$, since also $x+x^*$ commutes with $a_c$, by the spectral decomposition of compact self-adjoint operators, one can find a non-trivial (finite rank) spectral projection of $x+x^*$, which commutes with $a_c$, and thus $a_c$ would have an eigenvalue, leading to a contradiction. 
	It follows that $\delta_{a_c}:\Kc(\Hc)\to \Kc(\Hc)$ is bounded from below. 
	Thus, by \cite[Theorem 3.2]{Fi79}, one would have that $\sigma_l(a_c)\cap\sigma_r(a_c)=\emptyset$. 
	Since for 
	normal  operators, right and left spectra coincide with the spectrum, this implies that the spectrum of $a_c$ is empty, and therefore $\Hc_c=\{0\}$.
	
	It follows that the spectrum of $a$ has a dense subset of eigenvalues. Suppose that there are infinitely many eigenvalues. 
	By adding a multiple of the identity to $a$ (a change that does not affect $\delta_a$), we may suppose that $0$ is an accumulation point of the set of eigenvalues of $a$.
	From this infinite set one can select a sequence of (different) eigenvalues $\{\lambda_n: n\ge 1\}$ which are summable. 
	For each $n\ge 1$ pick a unit eigenvector $e_n$,  consider $\Hc_0$ the closed linear span of these eigenvectors, and denote $a_0=a\vert _{\Hc_0}\in \Bc(\Hc_0)$. 
	Note that $a_0$ is a compact, in fact, it is a nuclear operator. 
	It is apparent that since $\delta_a:\Kc(\Hc) \to \Kc(\Hc)$ has closed range, then  $\delta_{a_0}:\Kc(\Hc_0)\to \Kc(\Hc_0)$ also has closed range and by Lemma~\ref{splits}, it must have finite spectrum, a contradiction. 
	
	The proof for the case of the closure in $\Jc$ is similar and therefore omitted.
\end{proof}


In order to obtain a version of Theorem~\ref{noncompact} for non-self-adjoint operators $a\in\Bc(\Hc)^\nor$ (see Theorem~\ref{au_th} below), we now adapt some ideas from \cite{Vo76} to a suitable setting of operator ideals. 
A pair of symmetrically normed ideals $(\Ic,\Ic_1)$ will be hereafter called a \emph{dual pair of ideals} if $\Ic\Ic_1\subseteq\Sg_1(\Hc)$ and the bilinear functional 
\begin{equation}\label{duality}
\Ic\times\Ic_1\to\CC,\quad (x,y)\mapsto \Tr(xy)
\end{equation}
gives rise to an isometric isomorphism of the Banach space $\Ic$ onto the topological dual $\Ic_1^*$ of the Banach space $\Ic_1$. 
Such a dual pair is an admissible pair in the sense of \cite[Def. 9.22]{B06}  if one additionally assumes that $\Ic_1\subseteq\Ic$ and the Banach space $\Ic$ is reflexive. 
In this setting, we need the following slight variation of a \cite[Lemma 2.6]{Vo76}.

\begin{lemma}\label{au}
Let $(\Ic,\Ic_1)$ be a dual pair of symmetrically normed ideals in $\Bc(\Hc)$  and 
$a\in\Ic^\nor$.  
If $a'\in \Ic\cap \overline{\Oc_{\Bc(\Hc)}(a)}^{\Vert\cdot\Vert}$ 
then, 
for any symmetrically normed ideal $\Jc\subseteq\Bc(\Hc)$,  
the operator $\delta_a\vert_{\Jc^\sa}\colon\Jc^\sa\to\Ic$ has closed range if and only if the operator $\delta_{a'}\vert_{\Jc^\sa}\colon\Jc^\sa\to\Ic$ has closed range. 
\end{lemma}

\begin{proof}
The hypothesis on the operators $a$ and $a'$ is equivalent to the pair of conditions $a,a'\in\Ic^\nor$ and $\overline{\Oc_{\Bc(\Hc)}(a)}^{\Vert\cdot\Vert}=\overline{\Oc_{\Bc(\Hc)}(a')}^{\Vert\cdot\Vert}$, 
in which $a$ and $a'$ play symmetric roles. 
It therefore suffices to prove that if $\delta_a\vert_{\Jc^\sa}\colon\Jc^\sa\to\Ic$ has closed range, then the operator $\delta_{a'}\vert_{\Jc^\sa}\colon\Jc^\sa\to\Ic$ has closed range, too. 

By hypothesis, there exists $C>0$ such that for every $x\in\Jc^\sa$ there exists $x_0\in\Jc^\sa$ with $\Vert x_0\Vert_\Jc\le C\Vert\delta_a(x)\Vert_\Ic$ 
and $\delta_a(x)=\delta_a(x_0)$.   
Also, there exists a sequence of unitary operators $u_k\in\Bc(\Hc)$ with $\lim\limits_{k\to\infty}\Vert u_kau_k^*-a'\Vert=0$.
Then, for arbitrary $x\in\Jc^\sa$ and $k\ge 1$, there exists $x_k\in\Jc^\sa$ 
with  $\Vert x_k\Vert_\Jc\le C\Vert\delta_a(u_k^*xu_k)\Vert_\Ic$ 
and $\delta_a(u_k^*xu_k)=\delta_a(x_k)$. 

On the other hand, for every unitary operator $u\in\Bc(\Hc)$  and every $x\in\Ic$ we have 
$$\delta_a(u^*xu)=[u^*xu,a]=u^*[x,uau^*]u=u^*\delta_{uau^*}(x)u$$
while 
$$\Vert \delta_a(x)\Vert_\Ic
=\Vert xa-ax\Vert_\Ic
\le 2\Vert a\Vert \cdot \Vert x\Vert_\Ic$$
hence 
$$\lim\limits_{k\to\infty}\Vert \delta_{u_kau_k^*}(x)-\delta_{a'}(x)\Vert_\Ic=0 
\text{ for all }x\in\Ic.$$
Furthermore, 
\begin{align*}
\limsup_{k\to\infty}\Vert u_kx_ku_k^*\Vert_\Ic & = \limsup_{k\to\infty}\Vert x_k\Vert_\Jc \le C\limsup_{k\to\infty}\Vert\delta_a(u_k^*xu_k)\Vert_\Ic \\
&=C\limsup_{k\to\infty}\Vert\delta_{u_kau_k^*}(x)\Vert_\Ic =C\Vert \delta_{a'}(x)\Vert_\Ic.
\end{align*}
Since the unit ball of $\Ic\simeq\Ic_1^*$ is weakly compact, it follows in particular that, selecting a suitable subsequence, we may assume that there exists $y\in\Ic^\sa$ with 
\begin{equation}\label{au_proof_eq1}
\Vert y\Vert_\Ic\le C\Vert \delta_{a'}(x)\Vert_\Ic
\end{equation} 
and $\lim\limits_{k\to\infty}u_kx_ku_k^*=y$ in the 
dual weak topology of~$\Ic$. 
Since $\Fc(\Hc)\subseteq\Ic_1$, we have in particular 
$\lim\limits_{k\to\infty}u_kx_ku_k^*=y$ in the weak operator topology of~$\Bc(\Hc)$. 
This implies 
\begin{align*}
\delta_{a'}(y)-\delta_{a'}(x)
=&\lim\limits_{k\to\infty}[u_kx_ku_k^*,a']-[x,a'] \\ =&\lim\limits_{k\to\infty}u_k[x_k,u_k^*a'u_k]u_k^*-[x,a'] \\
=&\lim\limits_{k\to\infty}u_k[x_k,u_k^*a'u_k-a]u_k^*
+\lim\limits_{k\to\infty}u_k[x_k,a]u_k^*
-[x,a'] \\
=&\lim\limits_{k\to\infty}u_k[x_k,u_k^*a'u_k-a]u_k^*
+\lim\limits_{k\to\infty}u_k[u_k^*xu_k,a]u_k^*
-[x,a'] 
\\
=&\lim\limits_{k\to\infty}u_k[x_k,u_k^*a'u_k-a]u_k^*
+\lim\limits_{k\to\infty}[x,u_kau_k^*]
-[x,a'] \\
=&0
\end{align*}
hence $\delta_{a'}(y)=\delta_{a'}(x)$. 
Taking into account \eqref{au_proof_eq1} and the fact that $x\in\Jc^\sa$ is arbitrary, it follows that the  operator $\delta_{a'}\vert_{\Jc^\sa}\colon\Jc^\sa\to\Ic$ has closed range, 
and we are done. 
\end{proof}

\begin{theorem}\label{au_th}
	Let $(\Ic,\Ic_1)$ be a dual pair of symmetrically normed ideals in $\Bc(\Hc)$  and $a\in\Ic^\nor$.   
	Then, for any symmetrically normed ideal $\Jc\subseteq\Bc(\Hc)$,  
	if the operator $\delta_a\vert_{\Jc^\sa}\colon\Jc^\sa\to\Ic$ has closed range then $\spec(a)$ is finite. 
\end{theorem}

\begin{proof}
By Lemma~\ref{au}, the operator $\delta_{a'}\vert_{\Jc^\sa}\colon\Jc^\sa\to\Ic$ has closed range for every operator $a'$ in the operator norm closure $\overline{\Oc_{\Bc(\Hc)}(a)}^{\Vert\cdot\Vert}$ of the unitary orbit~$\Oc_{\Bc(\Hc)}(a)$. 
Since $a\in\Bc(\Hc)^\nor$, we have $\overline{\Oc_{\Bc(\Hc)}(a)}^{\Vert\cdot\Vert}\subseteq \Bc(\Hc)^\nor$ 
and $\spec(a')=\spec(a)$ for every $a'\in\overline{\Oc_{\Bc(\Hc)}(a)}^{\Vert\cdot\Vert}$. 
On the other hand, by the Weyl-von Neumann theorem as 
obtained in the special case of \cite[Prop. 2.1]{Vo76} for normal operators, 
there exists $a'\in \overline{\Oc_{\Bc(\Hc)}(a)}^{\Vert\cdot\Vert}$ diagonalizable. 
For such a diagonalizable operator $a'$ its spectrum is finite by Lemma~\ref{splits}\eqref{splits_item2}, hence $\spec(a)$ is finite as well. 
\end{proof}

Recall that when $a\in\Bc(\Hc)^\nor$ and $\spec(a)$ is finite, its collection of spectral projections is finite, thus its corresponding conditional expectation $E$ preserves any symmetrically normed ideal (Remark~\ref{conde}). 
Moreover, if $\Ac\subseteq\Bc(\Hc)$ is a $C^*$-algebra with $a\in\Ac^\nor$, then the spectral projections of $a$ belong to $\Ac$ hence the conditional expectation $E\colon\Bc(\Hc)\to\{a\}'$ satisfies $E(\Ac)=\Ac\cap\{a\}'$. 
If additionally $\1\in\Ac$, then we denote by $\U(\Ac):=\U(\Hc)\cap\Ac$ the unitary group of~$\Ac$, and also $\Oc_\Ac(a):=\{uau^*\mid u\in\U(\Ac)\}\subseteq\Ac$ and $\pi_\Ac^a\colon\U(\Ac)\to\Oc_\Ac(a)$, 
$u\mapsto uau^*$.

\begin{lemma}\label{finite}
If $a\in \Bc(\Hc)^{\nor}$ with finite spectrum,  
with its corresponding conditional expectation 
$E\colon\Bc(\Hc)\to\{a\}'$
then the following assertions hold true: 
\begin{enumerate}[{\rm(i)}]
	\item\label{finite_item1} 
	If $\Ac\subseteq\Bc(\Hc)$ is a $C^*$-algebra with $a\in\Ac$, then 
	the operator
	$\delta_a\vert_{\ie\Ac^\sa}\colon\ie\Ac^\sa\to\Ac$ has split closed  range and $\ie\Ac^\sa=\Ker(\delta_a\vert_{\ie\Ac^\sa})
	\dotplus\Ran((\id_{\Ac}-E)\vert_{\ie\Ac^\sa})$.  
	\item\label{finite_item2} 
	If $a\in\Fc(\Hc)^{\nor}$ and $\Ic,\Jc\subseteq\Bc(\Hc)$ are symmetrically normed ideals, then the operator
	$
	\delta_a\vert_{\ie\Jc^\sa}\colon\ie\Jc^\sa\to\Ic$ has split closed  range and $\ie\Jc^\sa=\Ker(\delta_a\vert_{\ie\Jc^\sa})
	\dotplus\Ran((\id_{\Bc(\Hc)}-E)\vert_{\ie\Jc^\sa})$.  
	Moreover, 
	$\Ran \delta_a=\Ker E\subseteq\Fc(\Hc)\subseteq \Ic\cap\Jc$. 
\end{enumerate}
\end{lemma}

\begin{proof}
\eqref{finite_item1} 
This can be obtained by a straightforward adaptation of the following proof of  Assertion~\eqref{finite_item2}, hence we skip the details.

\eqref{finite_item2} 
We use the notation from Remark~\ref{conde} again, 
so $\spec(a)\setminus\{0\}=\{\lambda_1,\dots,\lambda_n\}$, $\lambda_0:=0$, 
and $p_j$ is the orthogonal projection onto $\Ker(a-\lambda_j\id_\Hc)$ for $j=0,\dots,n$. 
Since $a\in\Fc(\Hc)$, we have $p_1,\dots,p_n\in\Fc(\Hc)$.  
For every $z\in\Bc(\Hc)$ we have 
\begin{equation}
\label{finite_proof_eq1}
z=\sum\limits_{i,j=0}^np_izp_j
\end{equation} 
and $E(z)=\sum\limits_{j=0}^np_jzp_j$, hence $z\in \Ker E$ if and only if $p_jzp_j=0$ for $j=0,\dots,n$, in particular $p_0zp_0=0$, and then 
 $z\in\Fc(\Hc)$ by~\eqref{finite_proof_eq1}. 
Thus 
$\Ker E\subseteq\Fc(\Hc)\subseteq\Ic\cap\Jc$.   
Moreover, since $E^2=E$, in general we have $\Ker E=\Ran(\id_{\Bc(\Hc)}-E)$. 
We have  $a=\sum\limits_{j=0}^n\lambda_jp_j$ hence, 
for any $z\in\Bc(\Hc)$, 
$$
\delta_a(z)=\sum_{i,j=0}^n (\lambda_i-\lambda_j)p_izp_j\in\Ker E.
$$
On the other hand, if $z\in \Ker E$ and we define 
$$
x:=\sum_{0\le i\ne j\le n} \frac{1}{\lambda_i-\lambda_j}p_i z p_j
\in\Ker E\subseteq \Fc(\Hc)\subseteq \Jc
$$
then it is easy to check that $\delta_a(x)=\sum\limits_{i\ne j}p_izp_j=z$, thus $z\in \Ran\delta_a$.

If we denote $z_{ij}:=p_izp_j$, then the equation $z^*=z$ is equivalent to $z_{ij}^*=z_{ji}$. 
Also, the equation $y=\delta_a(z)$ is equivalent to $y_{ij}=(\lambda_i-\lambda_j)z_{ij}$, which is further equivalent to $z_{ij}=(\lambda_i-\lambda_j)^{-1}y_{ij}$. 
Hence $y\in\delta_a(\Bc(\Hc)^\sa)$ if and only if 
$\overline{\lambda_i-\lambda_j}^{-1}y_{ij}^*=(\lambda_j-\lambda_i)^{-1}y_{ji}$, 
which is further equivalent to $y_{ij}^*=\alpha_{ji}y_{ji}$, 
where $\alpha_{ji}:=\overline{\lambda_i-\lambda_j}/(\lambda_j-\lambda_i)\in\TT$ 
for all $i\ne j$, while $y_{jj}=0$. 
We note that $\alpha_{ji}=\alpha_{ij}$. 
Selecting any $\beta_{ij}=\beta_{ji}\in\TT$ with $\beta_{ij}^2=\alpha_{ij}$, 
we further obtain that  $y\in\delta_a(\Bc(\Hc)^\sa)$ if and only if 
$y_{jj}=0$ and $(\beta_{ij}y_{ij})^*=\beta_{ji}y_{ji}$. 
Therefore, if we define 
$$\Vc_\pm:=\{y\in\Ic\mid (\beta_{ij}y_{ij})^*=\pm\beta_{ji}y_{ji}\text{ if }i\ne j\}$$
then $\Vc_\pm$ is a closed real linear subspace of $\Ic$ and we have the direct sum decomposition 
$$\Ic=
\Vc_+\dotplus\Vc_-$$ 
where 
$\Vc_+=\delta_a(\Bc(\Hc)^\sa)=\delta_a(\Jc^\sa)=\delta_a(\Fc(\Hc)^\sa)$. 

Finally, the kernel of the operator 
$\delta_a\vert_{\ie\Jc^\sa}\colon\ie\Jc^\sa\to\Ic$ 
is complemented since $\Ker(\delta_a\vert_{\ie\Jc^\sa})=\{a\}'\cap \ie\Jc^\sa$ 
hence 
$E\vert_{\ie\Jc^\sa}\colon \ie\Jc^\sa\to\ie\Jc^\sa$ is an idempotent operator 
with $E(\ie\Jc^\sa)=\Ker(\delta_a\vert_{\ie\Jc^\sa})$. 
\end{proof}

\begin{definition}\label{cross_def}
For any symmetrically normed ideals $\Jc,\Ic\subseteq\Bc(\Hc)$ and $a\in\Ic$,
we say that $\pi=\pi_{\Jc,\Ic}^a$ \emph{has local cross-sections}  if for every $b\in \Oc_{\Jc}(a)$ there exist an open neighborhood $V$ in the relative topology of $\Oc_{\Jc}(a)\subseteq\Ic$ and a continuous map $\sigma\colon V\to\U_{\Jc}$ with $\pi\circ\sigma=\id_V$ and $\sigma(b)=\1$.
\end{definition}

\begin{definition}\label{defsubma}
A {\it submanifold} of the Banach space~$\Ic$ is the image of any smooth injective map $f\colon M\to\Ic$, such that the image of the injective tangent map 
$$
T_mf\colon T_mM\to\Ic
$$
is a closed real subspace of~$\Ic$. 
We say that the manifold {\it splits} if $T_mf(T_mM)$ admits a direct complement. 
We say that the manifold is {\it embedded} if $f$ is an homeomorphism onto its image.

\begin{remark}\label{raeb}
In Definition~\ref{cross_def}, 
if $\pi$ is open and $\delta=\pi_{*\1}\colon \ie\Jc^\sa\to\Ic$ splits in range and kernel, the orbit $\Oc_{\Jc}(a)$ is a smooth (in fact, real analytic) submanifold of  $\Ic$. 
See for instance \cite[Prop. 1.5]{rae77} or \cite[Lemma 3.3.6]{La19}, 
and also Lemma~\ref{L336_enhanced}. 
\end{remark}

\begin{lemma}\label{L336_enhanced}
	Let $G$ be a Banach-Lie group with its Lie algebra~$\gg$, $\Xc$ a real Banach space, and $\pi\colon G\to\GL(\Xc)$ a morphism of Banach-Lie groups. 
	For fixed $x\in\Xc$ we define 
	$\pi^x\colon G\to\Xc$, $\pi^x(g):=\pi(g)x$. 
	Let  $\Oc(x):=\pi^x(G)$,  regarded as a topological subspace of~$\Xc$,
	and assume that the following conditions are satisfied:
	\begin{enumerate}
		\item[{\rm 1.}]\label{L336_enhanced_hyp1}
		There exist an open subset $D\subseteq\Xc$ and a continuous map $\tau\colon D\cap\Oc(x)\to G$ with $x\in D$, $\tau(x)=\1\in G$, and $\pi^x\circ \tau=\id_{D\cap\Oc(x)}$. 
		\item[{\rm 2.}]\label{L336_enhanced_hyp2}  
		The kernel and the range of the differential $\pi^x_{*,\1}\colon\gg\to \Xc$ are split closed subspaces of $\gg$ and $\Xc$, respectively. 
	\end{enumerate}
	Then the following assertions hold true: 
	\begin{enumerate}[{\rm(i)}]
		\item\label{L336_enhanced_item1} 
		The orbit $\Oc(x)$ is a split embedded submanifold of $\Xc$ and  $\pi^x\colon G\to\Oc(x)$ is a submersion. 
		\item\label{L336_enhanced_item2}  
		For every $y\in\Oc(x)$ there exist an open subset $V_y\subseteq\Xc$ 
		and a smooth mapping $\sigma^y\colon V_y\to G$ satisfying $y\in V_y$ and $\pi^x\circ\sigma^y\vert_{V_y\cap\Oc(x)}=\id_{V_y\cap\Oc(x)}$. 
		\item\label{L336_enhanced_item3} If there exists a norm $\Vert\cdot\Vert$ that defines the topology of~$\Xc$ and $\Vert\pi(g)v\Vert=\Vert v\Vert$ for every $v\in\Xc$,  
		then $\Oc(x)\subseteq \Xc$ 
		is a closed subset. 
	\end{enumerate}
\end{lemma}

\begin{proof}
	\eqref{L336_enhanced_item1} 
	The first hypothesis implies that the mapping $\pi^x\colon G\to\Oc(x)$ is open, 
	and then \cite[Lemma 3.3.6]{La19} is applicable.
	
	\eqref{L336_enhanced_item2} 
	Since $\pi^x\colon G\to\Oc(x)$ is a submersion and $\Oc(x)\subseteq\Xc$ is an embedded submanifold, there exist an open subset $\widetilde{V}_x\subseteq\Xc$ and a smooth mapping $\widetilde{\sigma}^x\colon \widetilde{V}_x\cap\Oc(x)\to G$ with $\pi^x\circ\sigma=\id_{\widetilde{V}_x\cap\Oc(x)}$ and $\sigma(x)=\1\in G$. 
	
	Now let $\Zc\subseteq\Xc$ be a closed linear subspace satisfying $\Xc=\Zc\dotplus \pi^x_{*,\1}(\gg)$. 
	Since $T_x(\Oc(x))=\Ran\pi^x_{*,\1}$ and $\Oc(x)\subseteq\Xc$ is a  submanifold, 
	there exist open subsets $V_x\subseteq \widetilde{V}_x$, $W_0\subseteq\Zc$, and $W'_0\subseteq \Ran\pi^x_{*,\1}$, and a diffeomorphism $\chi\colon V_x\to W_0\times W'_0$ with 
	$x\in V_x$, $\chi(x)=(0,0)\in W_0\times W'_0$, and $\chi(V_x\cap\Oc(x))=\{0\}\times W'_0$. 
	(See \cite[\S 2.9]{La19}.) 
	We now consider the Cartesian projection $\pr_2\colon W_0\times W'_0\to W'_0$, $(w,w')\mapsto w'$, hence we have the diagram
	$$\xymatrix{V_x \ar[d]_{\chi} \ar[r]^{\chi^{-1}\circ\pr_2\circ \chi\quad\ }&  V_x\cap\Oc(x) \ar[d]^{\chi}\ar@{^{(}->}[r] & \widetilde{V}_x\cap\Oc(x) \ar[r]^{\quad\widetilde{\sigma}^x} & 
		G\\
		W_0\times W'_0 \ar[r]^{\pr_2} & \{0\}\times W'_0 & & 
	}$$
	and then we define and $\sigma^x\colon V_x\to G$ by 
	$\sigma^x:=\widetilde{\sigma}^x\circ (\chi^{-1}\circ\pr_2\circ \chi)$. 
	
	Finally, for arbitrary $y\in\Oc(x)$ we select $g\in G$ with $y=\pi(g)x$ 
	and we define $V_y:=\pi(g)V_x\subseteq \Xc$ and $\sigma^y\colon V_y\to G$, 
	$\sigma^y:=\sigma^x\circ\pi(g)^{-1}\vert_{V_y}$.  
	
	\eqref{L336_enhanced_item3} 
	We must show that if $\lim\limits_{n\to\infty}\Vert y_n-y\Vert=0$ in $\Xc$ and $y_n\in\Oc(x)$ for every $n\ge 1$, then $y\in\Oc(x)$. 
	
	To this end, using the above notation, let $\varepsilon>0$ with $B(x,\varepsilon)\subseteq V_x$, where $B(x,\varepsilon_0):=\{v\in\Xc\mid\Vert x-v\Vert\le\varepsilon\}$. 
	Also let $n_0\ge 1$ with $\Vert y_n-y_m\Vert\le\varepsilon_0$ for all $m,n\ge n_0$. 
	Since $y_n\in\Oc(x)$, there exists $g_n\in G$ with $y_n=\pi(g_n)x$, 
	hence $\Vert x-\pi(g_{n_0})^{-1}y_n\Vert=\Vert \pi(g_{n_0}) x-y_n\Vert=\Vert y_{n_0}-y_n\Vert\le\varepsilon_0$ for all $n\ge n_0$. 
	That is, $\pi(g_{n_0})^{-1}y_n\in B(x,\varepsilon_0)\subseteq V_x$ for all $n\ge n_0$, and this also implies $\pi(g_{n_0})^{-1}y\in B(x,\varepsilon_0)\subseteq V_x$ since $y=\lim\limits_{n\to\infty}y_n$ in $\Xc$. 
	
	Now, let $\sigma\colon V_x\to G$ be the smooth mapping given by~\eqref{L336_enhanced_item2}, and define $g:=\sigma(\pi(g_0)^{-1}y)\in G$. 
	Since $\pi(g_0)^{-1}y=\lim\limits_{n\to\infty}\pi(g_0)^{-1}y_n$ in $V_x$ 
	and $\sigma$ is continuous, we obtain 
	$g=\lim\limits_{n\to\infty}\sigma(\pi(g_0)^{-1}y_n)$ in $G$. 
	Moreover, using $\pi^x\circ \sigma=\id_{V_x\cap\Oc(x)}$,  
	we have 
	$$\pi^x(g)=\lim\limits_{n\to\infty}\pi^x(\sigma(\pi(g_0)^{-1}y_n))
	=\lim\limits_{n\to\infty}\pi(g_0)^{-1}y_n=\pi(g_0)^{-1}y.$$
	Therefore $\pi(g_0)^{-1}y\in\Oc(x)$, and then $y\in\Oc(x)$, which completes the proof. 
\end{proof}

\begin{theorem}\label{finisemb}
If $a\in \Bc(\Hc)^{\nor}$ has finite spectrum, 
and $\Ac\subseteq\Bc(\Hc)$ is a $C^*$-algebra with $\1,a\in\Ac$, then 
$\Oc_\Ac(a)\subseteq\Ac$ is a closed embedded split submanifold 
and $\pi_\Ac^a\colon\U(\Ac)\to\Oc_\Ac(a)$ has local cross sections. 
If moreover $a\in\Fc(\Hc)^\nor$, then for any symmetrically normed ideals $\Jc,\Ic\subseteq\Bc(\Hc)$ the unitary orbit $\Oc_{\Jc}(a)\subseteq\Ic$  is a closed embedded split submanifold and $\pi_{\Jc,\Ic}^a\colon\U_\Jc\to\Oc_\Jc(a)$ has local cross sections.
\end{theorem}

\begin{proof}
We show that the hypotheses of Lemma~\ref{L336_enhanced} are satisfied. 
If $a\in \Ac^{\nor}$ has finite spectrum, 
then the operator $\delta_a\colon_{\ie\Ac^\sa}\colon \ie\Ac^\sa\to\Ac$ has split kernel and split closed range by Lemma~\ref{finite}\eqref{finite_item1}. 
Moreover, the mapping $\pi\colon \U(\Ac)\to\Oc_\Ac(a)$, $u\mapsto uau^*$, has local cross-sections by \cite[Prop. 4.6]{Fi78}. 
Hence $\Oc_\Ac(a)\subseteq\Ac$ is a closed embedded split submanifold by Lemma~\ref{L336_enhanced}. 

Now assume $a\in\Fc(\Hc)^\nor$ and let $\Jc,\Ic\subseteq\Bc(\Hc)$ be any symmetrically normed ideals. 
The operator $\delta_a\vert_{\ie\Jc^\sa}\colon\ie\Jc^\sa\to\Ic$ has split closed range and kernel by Lemma~\ref{finite}\eqref{finite_item2}. 
We will now construct a local cross-section for $\pi$ using the splitting of the range of $\pi_{*\1}$, 
following ideas in \cite{BR05} and \cite{AL10}. 
Since the range of $\delta_a\vert_{\ie\Jc^\sa}=\pi_{*\1}\colon\ie\Jc^\sa\to\Ic$ is closed and 
$\Ker(\delta_a\vert_{\ie\Jc^\sa})=\{a\}'\cap\ie\Jc^\sa
=\Ker((\id_{\Bc(\Hc)}-E)\vert_{\ie\Jc^\sa})$, it easily follows by the closed graph theorem that there exists a constant $C>0$ such that 
$$
\Vert z-E(z)\Vert _{\Jc}\le C\Vert \delta_a(z)\Vert _{\Ic}
$$
for all $z\in \ie\Jc^\sa$. 
Consider the relative open neighborhood of $a\in\Oc_\Jc(a)$
$$
V:=\{b\in \Oc_{\Jc}(a): \Vert b-a\Vert _{\Ic}<1/C\},
$$
and 
the open neighborhood of $\1\in\U_\Jc$,
$$W:=\{u\in\U_\Jc\mid uau^*\in V\}.$$
If $u=\1+x\in W\subseteq \U_{\Jc}$, then
\begin{align*}
\Vert \1-u^*E(u)\Vert
& =  \Vert u-E(u)\Vert \le \Vert u-E(u)\Vert _{\Jc} = \Vert x-E(x)\Vert _{\Jc}\le C\Vert xa-ax\Vert _{\Ic}\\
&= C\Vert ua-au\Vert _{\Ic}=C\Vert uau^*-a\Vert _{\Ic}<1.
\end{align*}
Hence 
$$
\Vert \1-u^*E(u)\Vert<1.
$$
thus $u^*E(u)\in\Bc(\Hc)$ is an invertible operator, hence $E(u)$ is invertible. 
Since $E(u)-\1=E(\1+x)-\1=E(x)\in \Jc$ therefore $E(u)\in \1+\Jc$. 
This implies that $\vert E(u)\vert =\sqrt{E(u)^*E(u)}\in \1+\Jc$ also (and it is still invertible). 
Hence, defining 
$$\Omega(E(u)):=E(u)\vert E(u)\vert ^{-1},$$ 
it is easy to check that $\Omega(E(u))$ is a unitary operator (in fact, this is the unitary operator in the polar decomposition of $E(u)$). 
It is also apparent that $\Omega(E(u))\in \1+\Jc$, therefore $\Omega(E(u)) \in U_{\Jc}$, and then $u^*E(u)\in U_{\Jc}$ as well. 
Moreover, since $u^*E(u)$ is close to $\1$ in uniform norm, it is plain 
that the map 
$$W\to\U_\Jc,\quad u\mapsto u^*E(u)$$ 
is continuous. 
Now we define $\sigma:V\to U_{\Jc}$
$$
\sigma(b)=u\Omega(E(u^*)), \quad \textrm{ when }b=uau^*\text{ with }u\in W.
$$
Note that if $b=uau^*=waw^*$ for $u,w\in W$, then $v:=u^*w$ commutes with $a$ and by the uniqueness of the polar decomposition for invertible operators, $w\Omega(E(w)^*)=uv\Omega(v^*E(w^*))=u\Omega(E(u^*))$ showing that $\sigma$ is well-defined. 
Using the continuous action of $\U_{\Jc}$, it suffices to prove that $\sigma$ is continuous at $a\in \Oc_{\Jc}(a)$ with respect to the topology inherited by $\Oc_\Jc(a)$ from~$\Ic$. 
To this end, assume then that $u_nau_n^*\to a$ in the topology of $\Ic$, with $u_n\in \U_\Jc$. 
Then $\Vert u_na-au_n\Vert _{\Ic}\to 0$ and as above, $\Vert \1-u_nE(u_n)\Vert _{\Jc}=\Vert u_n-E(u_n^*)\Vert _{\Jc}\to 0$. 
Using the continuity of $\Omega(\cdot)$ given by \cite[Prop. A.4]{Ne04}, we have 
$$
\sigma(u_nau_n^*)=u_n\Omega(E(u_n^*))=\Omega(u_nE(u_n^*))\to\1 
\text{ in }\U_\Jc.
$$
Finally, note that $E(u)^*=\1+E(x^*)
$ commutes with $a$, therefore $\Omega(E(u^*))$ also commutes with $a$ and
$$
\pi(\sigma(b))=\sigma(b)a\sigma(b)^*=u\Omega(E(u^*))a\Omega(E(u^*))u^*=uau^*=b
$$
which proves that $\sigma$ is a local cross-section for $\pi$. 
Hence $\Oc(a)\subseteq\Bc(\Hc)$ is a closed embedded split submanifold by Lemma~\ref{L336_enhanced}, and this completes the proof. 
\end{proof}

It is well-known that 
the unitary orbit of 
\textit{any} 
operator $a\in \Bc(\Hc)^{\nor}$ is closed 
in the operator norm topology 
if and only if $\spec(a)$  is a finite set. 
(See \cite[Lemma 2.3]{Fi75} or \cite[Prop. 2.4]{Vo76}.) 
Thus the  following theorem   deals with closures in other norms, and generalizes  \cite[Theorem 5.1]{GKMSu17} 
along the same lines. 
Another result of this type is obtained in Corollary~\ref{closed} by a completely different approach.

\begin{theorem}\label{isclosed}
For any symmetrically normed ideals $\Jc\subseteq\Bc(\Hc)$ and $\Ic\subsetneqq\Bc(\Hc)$ with 
$a\in\Ic^\nor$, 
the 
orbit $\Oc_{\Jc}(a)
\subseteq  \Ic
$ is closed in $\Ic$ 
if and only if $\spec(a)$ is finite.
\end{theorem}

\begin{proof}
	The ``if'' part of this theorem was proved in 	Theorem~\ref{finisemb}. Conversely, 
	we prove that if $a\in\Ic^\nor$ and $\spec(a)$ is infinite, then 
	\begin{equation}\label{isclosed_proof_eq1}
	\overline{\Oc_{\Fc(\Hc)}(a)}^{\Vert\cdot\Vert_\Ic}\setminus\Oc_{\Bc(\Hc)}(a)\ne\emptyset
	\end{equation} 
	This implies in particular 
	$\Oc_\Jc(a)\subsetneqq \overline{\Oc_\Jc(a)}^{\Vert\cdot\Vert_\Ic}$, 
	since otherwise  $\overline{\Oc_{\Fc(\Hc)}(a)}^{\Vert\cdot\Vert_\Ic}\subseteq \overline{\Oc_\Jc(a)}^{\Vert\cdot\Vert_\Ic} =\Oc_\Jc(a)\subseteq \Oc_{\Bc(\Hc)}(a)$, which is a contradiction with~\eqref{isclosed_proof_eq1}. 
	
	Since $a\in\Ic\subseteq\Kc(\Hc)$ and $\spec(a)$ is infinite, 
	it follows that $\spec(a)\setminus\{0\}$ can be labeled as a sequence of mutually distinct complex numbers convergent to $0$, and then we 
	may select a sequence of mutually distinct eigenvalues $\{x_n\}_{n\ge 1}$ with 
	$\sum\limits_{n\ge 1}\vert x_n\vert<\infty$. 
	Let us denote $Y:=\spec(a)\setminus\{x_n\mid n\ge 1\}$ 
	and $p_y:=\Ker(a-y\id_\Hc)$ for any $y\in Y$. 
	For every $n\ge 1$ we denote $p_n:=\Ker(a-x_n\id_\Hc)$ and 
	we select a rank-one projection $q_n\le p_n$. 
	 We then define 
	 $$a_0:=\sum_{n\ge 1}x_nq_n\in\Kc(\Hc)^\nor$$
	 hence 
	 $$a=a_0+\sum_{n\ge 1}x_n(p_n-q_n)+\sum_{y\in Y}yp_y.$$
	Let $p_0:=\sum_{n\ge1} q_n$,  $\Hc_0=\overline{\Ran(a_0)}=\Ran(p_0)$, and  $\Xc_n=\Ran(\sum\limits_{j=1}^{n+1} q_j)$ for $n\ge1$. 
	
	We now modify slightly $a_0$: for each $n\ge 1$ let
	$$
	a_0^n= x_{n+1}q_1+ \sum_{i=1}^n x_i q_{i+1}  + \sum_{i\ge n+2}x_i q_i 
	$$
	Note that $a_0^n$ is obtained from $a_0$ by a finite permutation of the basis of $\Xc_n$,  
	thus $a_0^n=u_na_0u_n^*$ with $u_n\in \U(\Hc)$ and $(\1-u_n)\Xc_n^\perp=\{0\}$. 
Then clearly 
	$u_n\in \U(\Hc)\cap(\1+\Fc(\Hc))\subseteq\U_\Jc$ and
	$$
	a_n:=u_n a u_n=a_0^n+(a-a_0)\in \Oc_{\Jc}(a) \quad \forall n\in\mathbb N.
	$$
	Now let $a_0':=\sum_{i\ge 1} x_i q_{i+1}$ and  $a':=a_0'+(a-a_0)$, and note that
	\begin{align*}
	\Vert a'-a_n\Vert_\Ic 
	& =\Vert a_0'-a_0^n\Vert_\Ic 
	=\Vert x_{n+1}q_1+\sum_{k\ge n} (x_{k+2}-x_{k+1})q_{k+2}\Vert_\Ic \\
	& \le \vert x_{n+1}\vert+\sum_{k\ge n} \vert x_{k+2}-x_{k+1}\vert \le 2\sum_{k\ge n+1} \vert x_k\vert ,
	\end{align*}
Letting $n\to \infty$  we obtain $\lim\limits_{n\to\infty}\Vert a'-a_n\Vert_\Ic=0$. 
	Thus $a'\in\overline{\Oc_\Jc(a)}^{\Vert\cdot\Vert_\Ic}$. 
	If $\Ker a=\{0\}$, then $a'\not\in \Oc_\Jc(a)$  
	since $\Ran q_1\subseteq\Ker a'$, and this proves that the orbit is not closed. If $\Ker a\ne\{0\}$, some modifications of the previous argument as in \cite[Th. 5.1]{GKMSu17} also show that the orbit is not closed.
\end{proof}

For any symmetrically normed ideals $\Jc$ and $\Ic$, 
we define the following subsets of~$\Ic^\nor$: 
\begin{align*}
{\rm LocCl}_{\Jc}(\Ic):=
&\{a\in\Ic^\nor\mid \Oc_{\Jc}(a) \text{ is locally closed in }\Ic\},\\
{\rm Cl}_{\Jc}(\Ic):=
&\{a\in\Ic^\nor\mid \Oc_{\Jc}(a) \text{ is closed in }\Ic\},\\
{\rm Cross}_{\Jc}(\Ic):=
&\{a\in\Ic^\nor\mid \pi_{\Jc,\Ic}^a \text{ has local cross-sections}\},\\
{\rm Sbm}_{\Jc}(\Ic):=
&\{a\in\Ic^\nor\mid \Oc_{\Jc}(a) \text{ is a submanifold of }\Ic\}.
\end{align*}
\end{definition}

\begin{remark}\label{remsubma}
In general, one has: $ {\rm Sbm}_{\Jc}(\Ic) \subseteq {\rm LocCl}_{\Jc}(\Ic)={\rm Cl}_{\Jc}(\Ic)$. 
For the last equality of sets, we use that  the action of $\Ic$ in $\Oc_{\Jc}(a)$ is transitive, for any  $\Jc$ and $\Ic$. 
Specifically, if $a\in\Ic^\nor$ and  $\Oc_{\Jc}(a)$ is locally closed in $\Ic$, there exists $\varepsilon>0$ such that the set $C:=\{b\in  \Oc_{\Jc}(a): \Vert b-a\Vert _{\Ic}\le\varepsilon\}$ is closed in~$\Ic$. 
If $b_n=u_nau_n^*\in \Oc_{\Jc}(a)$  with $u_n\in\U_{\Jc}$ and $\Vert b_n- b\Vert _{\Ic}\to 0$ as $n\to\infty$, pick $n_0\in\NN$ such that $\Vert b_n-b_m\Vert _{\Ic}\le\varepsilon$ for $n,m\ge n_0$ and rename 
$b_n':=u_{n_0}^*b_nu_{n_0}\to u_{n_0}^*bu_{n_0}=:b'$ as $n\to\infty$. 
Since $b_{n_0}'=a$ and $b_n'\in C$ for all $n\ge n_0$, we have $b'\in C$.  
In particular $b'\in\Oc_{\Jc}(a)$, hence $b'=uau^*$ for some $u\in U_{\Jc}$, and then $b=(u_{n_0}u)a(u_{n_0}u)^*\in \Oc_{\Jc}(a)$. 

The above reasoning involves only the norm of $\Ic$, hence 
the operator ideal $\Jc$ need not carry any norm. 
A more general conclusion can be obtained along these lines: 
Let  $U$ be a group and $U\times X\to X$, $(u,x)\mapsto u.x$ 
be an action on a Hausdorff uniform space~$X$ satisfying the condition 
that for every entourage $\Ec\subseteq X\times X$ in a basis of entourages and every $u\in U$ 
we have $(u.x,u.y)\in\Ec$ if $(x,y)\in\Ec$. 
(For instance,  any group action by isometries on a metric space.) 
Then for any $a\in X$ its orbit $U.a$ is a locally closed subset of $X$ if and only if it is a closed subset of $X$. 
In fact, if $U.a$ is locally closed, then there exists an entourage~$\Ec$ 
in for which the set $C:=U.a\cap\Ec(a)$ is closed in $X$, where $\Ec(a):=\{x\in X:(a,x)\in\Ec\}$. 
If $\{u_i\}_{i\in I}$ is a net in $U$ for which there exists $b\in X$ with $u_i.a\to b$ as $i\in I$, then there exists $i_\Ec\in I$ such that if $i,j\ge i_\Ec$ then $(u_j.a,u_i.a)\in\Ec$. 
This implies $(a,u_{i_\Ec}^{-1}u_i)\in\Ec$, that is, $u_{i_\Ec}^{-1}u_i\in\Ec(a)\cap U.a=C$ for al $i\ge i_\Ec$. 
On the other hand, since $u_i.a\to b$ as $i\in I$, we have 
$u_{i_\Ec}^{-1}u_i.a\to u_{i_\Ec}^{-1}.b$ as $i\in I$, 
hence $u_{i_\Ec}^{-1}.b\in C\subseteq U.a$. 
Therefore $b\in U.a$, and we are done. 
\end{remark}

Special cases of the following theorem can be found in several places in the earlier literature: \cite[Th. 4.5]{AS89},  \cite[Th. 3.2]{BR05}, \cite[Th. 4.37]{B06}, \cite[Th. 4.4]{AL10},  \cite[Lemma 1]{BoVa16}, \cite[Th. 5.1]{GKMSu17}. 
Moreover for the case of the maximal ideal of a given symmetric norm, the equivalence between finite spectrum and submanifold structure was settled in \cite{ChiDIyL13}.

\begin{theorem}\label{locl}
If $\Ic,\Jc\subseteq\Bc(\Hc)$ are symmetrically normed ideals, then  
$${\rm LocCl}_{\Jc}(\Ic)
={\rm Cl}_{\Jc}(\Ic)
={\rm Cross}_{\Jc}(\Ic)
={\rm Sbm}_{\Jc}(\Ic)
=\Fc(\Hc)^\nor\subseteq\Ic^\nor.$$ 
\end{theorem}

\begin{proof}
We have $ {\rm Sbm}_{\Jc}(\Ic) \subseteq {\rm LocCl}_{\Jc}(\Ic)={\rm Cl}_{\Jc}(\Ic)$ by Remark~\ref{remsubma}. 
Moreover, ${\rm Cl}_{\Jc}(\Ic)=\Fc(\Hc)^\nor$ by Theorem~\ref{isclosed} 
and $\Fc(\Hc)^\nor\subseteq {\rm Sbm}_{\Jc}(\Ic)\cap{\rm Cross}_{\Jc}(\Ic)$ 
by Theorem \ref{finisemb} and its proof. 

It remains to check that ${\rm Cross}_{\Jc}(\Ic)\subseteq  {\rm Cl}_{\Jc}(\Ic)$, and to this end we adapt the method of proof of \cite[Prop. 2.1]{Fi78}. 
We must prove that for arbitrary $a\in {\rm Cross}_{\Jc}(\Ic)$ and $b\in\overline{\Oc_\Jc(a)}^{\Vert\cdot\Vert_\Ic}$ one has $b\in \Oc_\Jc(a)$. 
Using a local cross section $\sigma\colon D\cap\Oc_\Jc(a)\to\U_J$ of $\pi=\pi_{\Jc,\Ic}\colon\U_\Jc\to\Oc_\Jc(a)\subseteq\Ic$, for a suitable open subset $D\subseteq\Ic$ with $a\in D$, it is easily shown that for every $n\ge 1$ there exists $\delta_n>0$ such that if $u\in\U_\Jc$ and $\Vert u^*au-a\Vert_\Ic<\delta_n$ then there exists $w\in\U_\Jc$ with $\Vert w-\1\Vert_\Jc<1/2^n$ and $w^*aw=u^*au$. 
Now let $\{u_k\}_{k\ge 1}$ be any sequence in $\U_\Jc$ with $\lim\limits_{k\to\infty}\Vert u_k^*au_k-b\Vert_\Ic=0$. 
Then we can inductively select a sequence $1\le k_1<k_2<\cdots$ 
such that, for every $n\ge 1$, if $k\ge k_n$ then $\Vert u_k^*au_k-b\Vert_\Ic<\delta_n/2$. 
Since $k_{n+1}>k_n$ we have 
$\Vert u_{k_{n+1}}^*au_{k_{n+1}}-u_{k_n}^*au_{k_n}\Vert_\Ic
\le \Vert u_{k_{n+1}}^*au_{k_{n+1}}-b\Vert_\Ic
+\Vert b-u_{k_n}^*au_{k_n}\Vert_\Ic<\delta_n/2+\delta_n/2=\delta_n$, 
hence 
$$\Vert u_{k_n}u_{k_{n+1}}^*au_{k_{n+1}}u_{k_n}^*-a\Vert_\Ic<\delta_n.$$
By the way $\delta_n$ was selected, there exists $w_n\in\U_\Jc$ with $\Vert w_n-\1\Vert_\Jc<1/2^n$ and $w_n^*aw_n= u_{k_n}u_{k_{n+1}}^*au_{k_{n+1}}u_{k_n}^*$. 
Defining $v_n:=u_{k_n}^*w_nu_{k_n}$, we have $\Vert v_n-\1\Vert_\Jc<1/2^n$ and $u_{k_{n+1}}^*au_{k_{n+1}}=v_n^*u_{k_n}^*au_{k_n}v_n$. 
Iterating this formula for $n,n-1,\dots,1$, we obtain 
\begin{equation}
\label{locl_proof_eq1}
u_{k_{n+1}}^*au_{k_{n+1}}=v_n^*v_{n-1}^*\cdots v_1^*u_{k_1}^*a
u_{k_1}v_1\cdots v_{n-1}v_n.
\end{equation}
Here $\sum\limits_{n\ge 1}\Vert v_n-\1\Vert_\Jc<\infty$ hence, 
an application of the criterion \cite[Ch. IX, App. II, Th. 2]{Bou06} for  convergence of infinite products in the Banach algebra $\CC\1+\Jc$
shows that there exists $v\in\CC\1+\Jc$ with $\lim\limits_{n\to\infty}\Vert v-v_1\cdots v_{n-1}v_n\Vert_\Jc=0$. 
Since $v_n\in\U_\Jc$ for every $n\ge 1$, we easily obtain $v\in\U_\Jc$ and, 
passing to the limit for $n\to\infty$ in \eqref{locl_proof_eq1}, 
we obtain $b=v^*u_{k_1}^*au_{k_1}v$, hence $b\in\Oc_\Jc(a)$. 
This completes the proof. 
\end{proof}

We now take a look at the differentiable structure of the orbit (as a quotient manifold, with a topology that is possibly finer than the topology of the orbit as subspace of $\Ic$). 
This structure can be constructed if $\Jc$ is am-closed.

\begin{definition} 
For any operator ideal $\Jc\subseteq\Bc(\Hc)$ and $a\in\Bc(\Hc)$ 
we define 
$$\U_{\Jc}(a):=\{u\in \U_{\Jc}: uau^*=a\},$$ 
the isotropy subgroup of the action $\pi$. 
If $\Jc$ is a symmetrically normed ideal, then clearly $\U_{\Jc}(a)$ is a closed subgroup of $\U_{\Jc}$.
\end{definition}

One can endow the unitary orbit $\Oc_{\Jc}(a)$ with the structure of a real analytic Banach manifold via the bijective map  
\begin{equation}\label{bij}
\U_{\Jc}/\U_{\Jc}(a)\to\Oc_{\Jc}(a),\quad u\U_\Jc(a)\mapsto uau^*
\end{equation} 
as follows. 

\begin{theorem}\label{manifold}
If the symmetrically normed ideal $\Jc$ is am-closed 
and $a\in\Kc(\Hc)^\nor$, 
then the following assertions hold true. 
\begin{enumerate}[{\rm(i)}]
\item\label{manifold_item1} $U_{\Jc}(a)$ is a Banach-Lie subgroup of $U_{\Jc}$ (with the inherited manifold topology).
\item\label{manifold_item2} The Lie algebra of $U_{\Jc}(a)$ is $\ker(\pi_{*1})\subset \ie\Jc^{sa} =\ug_{\Jc}$, and it is a split sub-algebra:
$$
\ug_{\Jc}=\lie(U_{\Jc}(a))\oplus \ker(E\vert _{\ug_{\Jc}}).
$$
\item\label{manifold_item3} The orbit $\Oc=\Oc_{\Jc}(a)$ is a real analytic manifold, and a homogeneous space for the smooth action of $A:U_{\Jc}\times \Oc\to \Oc$, $A(u,b)=ubu^*$.
\end{enumerate}
\end{theorem}

\begin{proof}
\eqref{manifold_item1} 
From the very definitions, it follows that $U_{\Jc}(a)$ is an algebraic subgroup of the group of invertible operators in $\CC\1+\Jc$ (see \cite{HK77} and \cite[Th. 4.13]{B06}), therefore  $U_{\Jc}(a)$ is a Banach-Lie group with its topology inherited from  $\CC\1+\Jc$, just like $U_{\Jc}$ does. 

\eqref{manifold_item2} We have  $\Ker((\id_{\Bc(\Hc)}-E)\vert_{\ie\Jc})=\Ker(\delta_a\vert_{\ie\Jc})=\lie(U_{\Jc}(a))$ by Lemma \ref{splits}\eqref{splits_item1} applied for $\Ic=\Jc$. 
But since we are now assuming that $\Jc$ is am-closed, we have $E(\Jc)\subseteq \Jc$ and since the range and kernel of $E$ are each other's direct complement ($E$ is an idempotent) the claim follows. 
Hence $U_{\Jc}(a)$ is a Banach-Lie subgroup of $U_{\Jc}$ with split sub-algebra.

\eqref{manifold_item3} 
 Assertion~\eqref{manifold_item2} directly implies that $\U_{\Jc}/\U_{\Jc}(a)$ is a smooth homogeneous space for the action described above (see \cite[Th. 4.19]{B06}), 
 and then the orbit $\Oc_\Jc(a)$ is turned into a smooth homogeneous space via the bijection~\eqref{bij}.
\end{proof}

\begin{remark}\label{amclosed}
In Theorem~\ref{manifold}, the orbit $\Oc_{\Jc}(a)$ has a smooth manifold structure when $\Jc$ is am-closed, regardless of~$\spec(a)$. 
Is that hypothesis on $\Jc$ really necessary? 
It is only used to find a direct complement to the Lie algebra of the isotropy group.
\end{remark}

\subsection{Dual pairs}\label{dualpairs}

We now prove a version of the above Theorem~\ref{manifold} in which the 
normal operator $a$ is not necessarily compact, 
in the setting provided by dual pairs of ideals as already used in Lemma~\ref{au}.

The following lemma is needed in the proofs of Proposition~\ref{manifold_bdd} and Lemma~\ref{BR05_Th2.2}. 

\begin{lemma}\label{averaging}
Let $\Xc$ and $\Xc_1$ be real Banach spaces endowed with a bilinear functional 
$\langle\cdot,\cdot\rangle\colon\Xc\times\Xc_1\to\RR$ 
that gives rise gives rise to a topological isomorphism of the Banach space $\Xc$ onto the topological dual of the Banach space $\Xc_1$. 
Moreover, we assume that a norm that defines the topology of $\Xc_1$ has been fixed, and $(S,+)$ is an abelian semigroup with a mapping 
$\alpha\colon S\to\Bc(\Xc)$, $s\mapsto \alpha^s$, satisfying the following conditions: 
\begin{itemize}
	\item For all $s,t\in S$ one has $\alpha^{s+t}=\alpha^s\alpha^t$. 
	\item One has $M:=\sup\limits_{s\in S}\Vert\alpha^s\Vert<\infty$. 
	\item For every $s\in S$ there exists $\alpha^s_*\in\Bc(\Xc_1)$ with $\langle \alpha^s(x),x_1\rangle=\langle x,\alpha^s_*(x_1)\rangle$ for all $x\in\Xc$ and $x_1\in\Xc_1$. 
\end{itemize}
If we denote $\Xc^S:=\{x\in\Xc\mid (\forall s\in S)\quad \alpha^s(x)=x\}$, 
then there exists $\widetilde{E}\in\Bc(\Xc)$ with $\widetilde{E}^2=\widetilde{E}$, $\Vert \widetilde{E}\Vert\le M$, and $\widetilde{E}(\Xc)=\Xc^S$. 
\end{lemma}

\begin{proof}
We adapt the method of proof of \cite[Th. 3.1(a)]{BePr07} to \emph{real} Banach spaces. 
To this end let $D\colon \Xc\to\Xc_1^*$, $x\mapsto \langle x,\cdot\rangle$, be the topological isomorphism that exists by hypothesis. 
Also let $\ell^\infty_{\CC}(S)$ be the commutative $C^*$-algebra consisting of all bounded functions $\xi\colon S\to\CC$, and for every $t\in S$ define $L_t\colon\ell^\infty_{\CC}(S)\to\ell^\infty_{\CC}(S)$, $(L_t\xi)(s):=\xi(t+s)$. 
Since $S$ is an abelian semigroup, it is well known that it has an invariant mean. 
(See \cite[Th. 2$(\alpha)$]{Di50}.) 
That is, the $C^*$-algebra $\ell^\infty_{\CC}(S)$ has a state $\mu\colon \ell^\infty_{\CC}(S)\to\CC$ satisfying $\mu\circ L_t=\mu$ for all $t\in T$. 

Now let $\ell^\infty_{\Xc}(S)$ be the real Banach space of all bounded functions $f\colon S\to\Xc$ and, using the hypothesis, note that the bounded linear operator 
$$\iota\colon\Xc\to \ell^\infty_{\Xc}(S), \quad (\iota(x))(s):=\alpha^s(x)$$ is well defined and $\Vert\iota\Vert\le M$. 
Moreover, for arbitrary $t\in S$ and $x_1\in\Xc_1$ one has 
$\langle(\iota(x))(\cdot),x_1\rangle\in \ell^\infty_{\CC}(S)$ 
and 
\begin{equation}\label{averaging_proof_eq1}
L_t(\langle(\iota(x))(\cdot),x_1\rangle)
=\langle(\iota(x))(t+\cdot),x_1\rangle
=\langle \alpha^t(\iota(x))(\cdot),x_1\rangle.
\end{equation} 
We then define 
$$\widetilde{E}\colon \Xc\to\Xc,\quad 
\langle \widetilde{E}(x),x_1\rangle:=\mu(\langle(\iota(x))(\cdot),x_1\rangle).$$
Since $D\colon \Xc\to\Xc_1^*$ is a topological isomorphism, it follows that $\widetilde{E}\colon\Xc\to\Xc$ is a bounded linear operator with $\Vert \widetilde{E}\Vert\le M$. 

If $x\in \Xc^S$, then $\iota(x)\colon S\to \Xc$ is the constant function that is equal to $x$ everywhere on $S$, hence, using $\mu(\1)=1$ 
(where $\1\in\ell^\infty_{\CC}(S)$ is the unit element), we easily obtain $\widetilde{E}(x)=x$. 
On the other hand, for arbitrary $t\in S$, using the property $\mu\circ L_t=\mu$, 
one obtains 
\begin{align*}
\langle \alpha^t(\widetilde{E}(x)),x_1\rangle
&=\langle \widetilde{E}(x),\alpha^t_*(x_1)\rangle \\
&=\mu(\langle(\iota(x))(\cdot),\alpha^t_*(x_1)\rangle) \\
&=\mu(\langle \alpha^t((\iota(x))(\cdot)),x_1\rangle) \\
&=\mu(L_t(\langle(\iota(x))(\cdot),x_1\rangle)) \\
&=\mu(\langle(\iota(x))(\cdot),x_1\rangle) \\
&= \langle \widetilde{E}(x),x_1\rangle
\end{align*}
where for the fourth and fifth equalities we have used \eqref{averaging_proof_eq1} and $\mu\circ L_t=L_t$, respectively. 
Consequently $\alpha^t\circ \widetilde{E}=\widetilde{E}$. 
Thus $\widetilde{E}(\Xc)\subseteq \Xc^S$ and $\widetilde{E}(x)=x$ for every $x\in\Xc^S$, 
and then $\widetilde{E}(\widetilde{E}(x))=x$ for every $x\in\Xc$ and moreover $E(\Xc)=\Xc^S$, which completes the proof. 
\end{proof}

The following fact is a simultaneous extension of \cite[Th. 9.29(ii)]{B06} from admissible pairs to dual pairs and of \cite[Prop. 4.5]{BR05} to the non-reflexive case.

\begin{proposition}\label{manifold_bdd}
If $(\Jc,\Jc_1)$ is a dual pair of ideals and $a\in\Bc(\Hc)^\nor$, 
then the following assertions hold. 
\begin{enumerate}[{\rm(i)}]
\item\label{manifold_bdd_item1} There exists a bounded linear operator $\widetilde{E}\colon \ug_{\Jc}\to \ug_{\Jc}$ with $\widetilde{E}^2=\widetilde{E}$ and $\widetilde{E}(\ug_{\Jc})=\lie(U_{\Jc}(a))$, thus
$$
\ug_{\Jc}=\lie(U_{\Jc}(a))\oplus \Ker\widetilde{E}.
$$
\item\label{manifold_bdd_item2} $U_{\Jc}(a)$ is a Banach-Lie subgroup of $U_{\Jc}$. 
\end{enumerate}
\end{proposition}

\begin{proof}
	\eqref{manifold_bdd_item1} 
For every $(s,t)\in\RR^2$ we define  
$$\alpha_{(s,t)}\colon\Bc(\Hc)\to \Bc(\Hc), \quad \alpha_{(s,t)}(x)=\ee^{\ie (sa+ta^*)}x\ee^{-\ie  (sa+ta^*)},$$  
and it is clear that $\alpha_{(s,t)}\vert_{\Jc}=(\alpha_{(s,t)}\vert_{\Jc_1})^*$ with respect to the duality $\Jc\simeq\Jc_1^*$ given by \eqref{duality}. 
Moreover, 
$$\lie(U_{\Jc}(a))=\{x\in\ug_{\Jc}\mid (\forall (s,t)\in\RR^2)\quad  \alpha_{(s,t)}(a)=a\}.$$
It then follows by Lemma~\ref{averaging} applied for the additive (semi-)group $S=(\RR^2,+)$ that there exists a bounded linear operator $\widetilde{E}\colon \ug_{\Jc}\to \ug_{\Jc}$ with $\widetilde{E}^2=\widetilde{E}$ and $\widetilde{E}(\ug_{\Jc})=\lie(U_{\Jc}(a))$, hence $\Ker \widetilde{E}$ is a direct complement to the Lie algebra of $U_{\Jc}(a)$ in the Lie algebra of $U_{\Jc}$. 

\eqref{manifold_bdd_item2}
This has the same proof as in Theorem~\ref{manifold}. 
\end{proof}

Regarding examples to which Proposition~\ref{manifold_bdd} applies, 
we mention that for every symmetric norming function $\Phi$ with its dual symmetric norming function $\Phi^*$, the pair of symmetrically normed operator ideals $(\Sg_{\Phi^*},\Sg_\Phi^{(0)})$ is a dual pair in the above sense. 
(See for instance \cite[Rem. 4.2(iii)]{BR05} and the references therein.)

With the same proof as in Theorem \ref{manifold}, and by means of Proposition \ref{manifold_bdd}, we obtain the differentiable structure of the orbit as a corollary:

\begin{corollary}
If $(\Jc,\Jc_1)$ is a dual pair of ideals and $a\in\Bc(\Hc)^\nor$, then the orbit $\Oc:=\Oc_{\Jc}(a)$ is a real analytic Banach manifold, and a homogeneous space for the smooth action of $A:U_{\Jc}\times \Oc\to \Oc$, $A(u,b)=ubu^*$.
\end{corollary}

\section{Closures of unitary orbits}
\label{sec:closure}

In this section, we examine the closures of unitary orbits for the different topologies, namely the ones defined by the operator norm or by various symmetric ideal norms (Theorem~\ref{closures_th}). 
We recall that the ideal of finite rank operators $\Fc(\Hc)$ is dense in any separable symmetrically normed ideal $\Ic\subseteq\Bc(\Hc)$.

\begin{remark}\label{normclosure}
If $a,b\in\Bc(\Hc)^{\nor}$, then $b\in\overline{\Oc_{\Bc(\Hc)}(a)}^{\Vert\cdot\Vert}$ if and only if the following conditions are satisfied: 
	\begin{enumerate}[{\rm(i)}]
		\item One has $\spec(a)=\spec(b)$. 
		\item Each isolated point of $\spec(a)$ has the same multiplicity as an eigenvalue of~$a$ and as an eigenvalue 
		of~$b$.
	\end{enumerate}
	These conditions are further equivalent to the condition that if $E^a(\cdot)$ and $E^b(\cdot)$ are the spectral measures of $a$ and $b$, respectively, then for every open subset $D\subseteq\CC$ one has $\rank E^a(D)=\rank E^b(D)$. See {\cite[Th.\ 1]{GePa74}, \cite[Rem.\ on page 219]{Fi77}, \cite[Th.\ 1.1]{Sh07}} for proofs of these facts.
\end{remark}

\begin{definition}\label{rangeproj}
	For every operator $a\in\Bc(\Hc)^{\nor}$, we denote by $p_a$  the orthogonal projection onto the closure of the range of $a$, that is $p_a=p_{\overline{\Ran(a)}}$. 
	Since $a\in\Bc(\Hc)^{\nor}$, we have 
	\begin{equation}
	\label{rangeproj_eq1}
	p_a=p_{\vert a\vert }=E^a(\CC)
	\end{equation}
	where $E^a(\cdot)$ is the spectral measure of~$a$.  
	It will be useful to consider the set of partial isometries $\Vc(\Hc)$ acting in $\Hc$, and also the subset of partial isometries with fixed initial space 
$$
\Vc_a=\{v\in \Vc(\Hc): v^*v=p_a\},
$$
a closed subset of $\Bc(\Hc)$.
\end{definition}

\begin{lemma}\label{grpd}
Let $a\in \Kc(\Hc)^{\nor}$. 
Then  $b\in\overline{\Oc_{\Bc(\Hc)}(a)}^{\Vert\cdot\Vert}$ if and only if there exists $v\in \Vc_a$ with $b=vav^*$. 
In that case, $a=v^*bv$ and $vv^*=p_b$.
\end{lemma}

\begin{proof}
If $b\in\overline{\Oc_{\Bc(\Hc)}(a)}^{\Vert\cdot\Vert}$ then 
$\rank E^a(\CC)=\rank E^b(\CC)$ by Remark~\ref{normclosure}. 
Hence, by~\eqref{rangeproj_eq1}, there exists an isometry $v$ from the (closure of) the range of $a$ to the (closure of) the range of $b$; extending it as $0$ on $(\Ran a)^\perp=\Ker a^*=\Ker a$, we obtain a partial isometry such that $v^*v=p_a$ and $b=vav^*$. 

Conversely, if $b=vav^*$ for some partial isometry $v$ with $v^*v=p_a$, and $vv^*=p_b$ then for each eigenvalue of $a$, if $p_k$ is the corresponding eigen-projection, let $q_k=vp_kv^*$. 
Clearly $bq_k=vav^*vp_kv^*=vap_kv^*=\lambda_k q_k$.  
Thus, the operator $b$ has the same eigenvalues with the same multiplicity as $a$, and by 
Remark~\ref{normclosure}, 
we obtain $b\in\overline{\Oc_{\Bc(\Hc)}(a)}^{\Vert\cdot\Vert}$. 
Finally, since $v^*bv=v^*vav^*v=a$, the range of $vv^*$ is the closure of the range of~$b$.
\end{proof}

In the notation of Lemma~\ref{grpd}, 
$$
\Ker v=\Ker a, \; \Ran v=\overline{\Ran b}, \quad \Ker v^*=\Ker b,\; \Ran v^*=\overline{\Ran a},
$$
and then the uniform norm closure of the full unitary orbit of $a$ if obtained by acting with partial isometries $v$ with \textit{initial space} equal to the (closure of) the range of $a$.

\begin{proposition}\label{9june2018}
If $\Ic\subsetneqq\Bc(\Hc)$ is an operator ideal, and $a\in\Ic^\nor$, then   we have 
$\overline{\Oc_{\Bc(\Hc)}(a)}^{\Vert\cdot\Vert}\subseteq\Ic$. 
\end{proposition}

\begin{proof}[Proof 1]
Use Lemma~\ref{grpd}. 
\end{proof}

\begin{proof}[Proof 2]
Let $u_n\in\Bc(\Hc)$ be a sequence of unitary operators with  $\lim\limits_{n\to\infty}u_n^*au_n=b$ in the operator norm topology. 
Then, by the well-known continuity property of the functional calculus with respect to sequences of normal operators, for every continuous function $f\colon\CC\to\CC$ one has 
$$
\lim\limits_{n\to\infty}u_n^*f(a)u_n=\lim\limits_{n\to\infty}f(u_n^*au_n)=f(b)
$$
in the operator norm topology.  In particular, for $f(\cdot)=\vert\cdot\vert$, we obtain $\vert b\vert\in\overline{\Oc_{\Bc(\Hc)}(\vert a\vert)}^{\Vert\cdot\Vert}$. It then follows by Remark~\ref{normclosure} that the sequences of singular numbers of the operators $a$ and $b$ coincide: $s_n(b)=s_n(a)$ for every $n\ge 1$. Since $a\in\Ic$, it now follows that $b\in\Ic$, as explained at the top of \cite[page 26]{Sch60}. 
\end{proof}

To study the unitary orbits for the action of the groups $\U_{\Jc}$, we begin by noting that two finite dimensional subspaces are conjugated by a special unitary if they have the same dimension. 

\begin{lemma}\label{8june2018}
For any linear subspaces $\Ec_1,\Ec_2\subseteq\Hc$ with $\dim\Ec_1=\dim\Ec_2<\infty$, 
there exists $u\in\U_{\Fc(\Hc)}(\Hc)$ with $u\Ec_1=\Ec_2$. 
\end{lemma}

\begin{proof}
Let $\Ec:=\Ec_1+\Ec_2$. 
Since $\dim\Ec_1=\dim\Ec_2$, there exists $v_0\in\U(\Ec)$ with $v_0(\Ec_1)=\Ec_2$. 
Defining $u\in\U(\Hc)$ by $u\vert_{\Ec_1}=v_0$ and $u=\id_\Hc$ on $\Ec_1^\perp$ ($\subseteq\Hc$), 
one has $u\in\U_{\Fc(\Hc)}(\Hc)$ since $\dim\Ec_1<\infty$, and on the other hand $u\Ec_1=\Ec_2$. 
\end{proof}

\begin{theorem}\label{closures_th}
For any symmetrically normed ideal  $\Ic\subsetneqq\Bc(\Hc)$,  
the following conditions are equivalent: 
\begin{enumerate}[{\rm(i)}]
	\item\label{closures_th_item1} The ideal $\Ic$ is separable. 
	\item\label{closures_th_item2} For every  $a\in\Ic^\nor$ one has 
	$\overline{\Oc_{\Bc(\Hc)}(a)}^{\Vert\cdot\Vert}\subseteq \overline{\Oc_{\Fc(\Hc)}(a)}^{\Vert\cdot\Vert_\Ic}$.
\item\label{closures_th_item3} For every  operator $a\in\Ic$ with $a\ge 0$ one has 
$\Oc_{\Bc(\Hc)}(a)\subseteq \overline{\Oc_{\Fc(\Hc)}(a)}^{\Vert\cdot\Vert_\Ic}$.
\end{enumerate}
If any of these conditions is satisfied, then for any $a\in\Ic^\nor$ one has 
$$\overline{\Oc_{\Fc(\Hc)}(a)}^{\Vert\cdot\Vert_\Ic}
=\overline{\Oc_{\Ic}(a)}^{\Vert\cdot\Vert_\Ic}
=\overline{\Oc_{\Bc(\Hc)}(a)}^{\Vert\cdot\Vert_{\Ic}}
=\overline{\Oc_{\Bc(\Hc)}(a)}^{\Vert\cdot\Vert}.$$
\end{theorem}

\begin{proof}
The last assertion follows since $\Fc(\Hc)\subseteq \Ic\subseteq\Bc(\Hc)$ and $\Vert\cdot\Vert\le\Vert\cdot\Vert_\Ic$, and then it is easily seen that 
$$\overline{\Oc_{\Fc(\Hc)}(a)}^{\Vert\cdot\Vert_\Ic}
\subseteq \overline{\Oc_{\Ic}(a)}^{\Vert\cdot\Vert_\Ic}
\subseteq \overline{\Oc_{\Bc(\Hc)}(a)}^{\Vert\cdot\Vert_{\Ic}}
\subseteq \overline{\Oc_{\Bc(\Hc)}(a)}^{\Vert\cdot\Vert} 
\subseteq\Ic$$
where the last inclusion follows by Proposition~\ref{9june2018}.  

\eqref{closures_th_item1}$\Rightarrow$\eqref{closures_th_item2}: 
We prove that for arbitrary 
$b\in 
\overline{\Oc_{\Bc(\Hc)}(a)}^{\Vert\cdot\Vert}$ one has 
$b\in \overline{\Oc_{\Fc(\Hc)}(a)}^{\Vert\cdot\Vert_\Ic}$. 
Since $a,b\in\Kc(\Hc)$, 
it follows by Remark~\ref{normclosure} that there exists $n_0\in\NN\cup\{\aleph_0\}$ 
with $$\spec(a)\setminus\{0\}=\spec(b)\setminus\{0\}=\{\lambda_n\mid 1\le n< n_0\}$$
where $\lambda_n\ne\lambda_m$ if $n\ne m$. 
Moreover, if $1\le n<n_0$, it follows by Remark~\ref{normclosure} along with Lemma~\ref{8june2018} that there exists $v_n\in \U_{\Fc(\Hc)}$ with 
$v_n(E^a(\lambda_k)\Hc)=E^b(\lambda_k)$ for $k=1,\dots,n$. 
It then follows that $v_n((E^a(\lambda_k)\Hc)^\perp)=(E^b(\lambda_k))^\perp$ (since $v_n$ is unitary) and then 
$$v_n E^a(\lambda_k) v_n^*=E^b(\lambda_k)\text{ if }1\le k\le n.$$
Denoting $p_n:=\sum\limits_{k=1}^n E^a(\lambda_k)$ and $q_n:=\sum\limits_{k=1}^n E^b(\lambda_k)$, 
we have 
$v_n(ap_m)v_n^*=bq_m$, hence 
\begin{equation}\label{closures_th_proof_eq1}
v_nav_n^*-b=v_n(a-ap_m)v_n^*+(bq_m-b) \text{ if }n\ge m.
\end{equation}
Since $a,b\in\Ic$ and $\Ic$ is a separable symmetrically normed ideal, it follows by \cite[Ch. III, Th. 6.3]{GoKr69} 
that for arbitrary $\varepsilon>0$ there exists $m_\varepsilon\ge 1$ with 
$\Vert a-ap_{m_\varepsilon}\Vert_\Ic<\varepsilon$ and $\Vert b-bq_{m_\varepsilon}\Vert_\Ic<\varepsilon$. 
Then, by \eqref{closures_th_proof_eq1}, 
for every $n\ge m_\varepsilon$ one has  
$\Vert v_nav_n^*-b\Vert_\Ic<2\varepsilon$. 
Consequently $\lim\limits_{n\to\infty}\Vert v_nav_n^*-b\Vert_\Ic=0$, 
which shows that $b\in \overline{\Oc_{\Fc(\Hc)}(a)}^{\Vert\cdot\Vert_\Ic}$, 
as claimed. 

\eqref{closures_th_item2}$\Rightarrow$\eqref{closures_th_item3}:  
Obvious. 

\eqref{closures_th_item3}$\Rightarrow$\eqref{closures_th_item1}: 
We prove that if $\Ic$ is not separable, then 
there exists $a\in\Ic$ with $a\ge 0$ and 
$\Oc_{\Bc(\Hc)}(a)\not\subseteq \overline{\Oc_{\Fc(\Hc)}(a)}^{\Vert\cdot\Vert_\Ic}$. 

Since $\Ic$ is a symmetrically normed ideal, one has $\Sg_1(\Hc)\subseteq\Ic$, and moreover, since $\Ic$ 
is not separable, we actually have $\Sg_1(\Hc)\subsetneqq\Ic$. 
(See \cite[Ch. III, \S 2, pages 69--70]{GoKr69}.)
Hence 
there exists $x_0\in \Ic$ with $x_0\ge 0$, whose sequence of singular values $(s_k(x_0))_{k\ge 1}$ is 
bi-normalizing, that is, $\lim_{k\to\infty}s_k(x_0)=0$ and $\sum_k s_k(x_0)=\infty$. 

Let $\{e_k\}_{k\ge 1}$ be an orthonormal basis of $\Hc$, 
and $\{e_k\otimes\overline{e_k}
\}_{k\ge 1}$ be the corresponding sequence of rank-one orthogonal projections. 
For $x\in \Kc(\Hc)$ let
$$
\Vert x\Vert _{\Ic_0}=\sup\left\{ \frac{\sum_{k=1}^n s_k(x)}{\sum_{k=1}^n s_k(x_0)}: n\in\mathbb N\right\}
$$
which induces a symmetric norm in $\Kc(\Hc)$ \cite[Ch. III, Th.  14.1]{GoKr69}, and we let $\Ic_0=\{x\in \Kc(\Hc): \Vert x\Vert _{\Ic_0}<\infty\}$. 
Let $a=\sum\limits_{k\ge 1} s_k(x_0) 
e_{2k}\otimes\overline{e_{2k}}$ 
and 
$b=\sum\limits_{k\ge 1} s_k(x_0) 
e_{2k-1}\otimes\overline{e_{2k-1}}$. 
We have  
$b=uau^*$ for 
unitary operator $u\in \Bc(\Hc)$ satisfying $ue_{2k}=e_{2k-1}$, $ue_{2k-1}=e_{2k}$ for $k\ge 1$, hence $b\in \Oc_{\Bc(\Hc)}{(a)}$. 
Since the sequences of singular values of $a$, $b$, and $x_0$ coincide and $x_0\in\Ic$,  
it is also easy to check that $a,b\in\Ic$. 
(Cf. \cite[page 26]{Sch60}.)
Moreover
$$
\lim\limits_{n\to \infty} \frac{\sum_{k=1}^n s_k(a)}{\sum_{k=1}^n s_k(x_0)}=1
$$
and likewise with $b$ and $b-a$. This shows \cite[loc. cit]{GoKr69} that neither of $a,b,a-b$ belongs to $\Ic_0^m$, which is the closure of finite rank operators in the norm of~$\Ic_0$. 
Assume $b\in \overline{\Oc_{\Fc(\Hc)}(a)}^{\Vert\cdot\Vert_\Ic}$. 
Then $b=\lim_n u_nau_n^*$ with $u_n=\1+x_n$ and $x_n\in\Fc(\Hc)$, implying
$$
b -a=\lim_n x_na+ax_n^*+x_nax_n^*
$$
which is impossible since $x_na+ax_n^*+x_nax_n^*\in \Fc(\Hc)$ for all $n\ge 1$, 
while we have seen above that $b-a\not\in \Ic^m_0$. 
Thus $b\in \Oc_{\Bc(\Hc)}{(a)}\setminus \overline{\Oc_{\Fc(\Hc)}(a)}^{\Vert\cdot\Vert_\Ic}$, and we are done. 
\end{proof}

We now obtain the following generalization of \cite[Th. 5.1]{GKMSu17}. 

\begin{corollary}\label{closed}
Let $\Ic\subseteq\Bc(\Hc)$ be a separable symmetrically normed ideal. 
If $a\in\Ic^\nor$, then $\Oc_{\Bc(\Hc)}(a)$ is a closed subset of $\Ic$ if and only if $a\in\Fc(\Hc)$. 
\end{corollary}

\begin{proof}
If $a\in\Fc(\Hc)$ then $\Oc_{\Bc(\Hc)}(a)$ is a closed subset of $\Ic$ by Theorem~\ref{finisemb}.

Now assume  $a\in\Ic\setminus\Fc(\Hc)$. 
Since $a$ is a compact operator, it follows that the closure of its range is a closed infinite-dimensional subspace $\Hc_a\subseteq\Hc$. 
Since $\Hc_a$ is infinite-dimensional, there exists a closed infinite-dimensional subspace $\Hc_b$ for which the orthogonal complements $\Hc_a^\perp$ and $\Hc_b^\perp$ have differing dimensions, 
and in particular no unitary operator in $\U(\Hc)$ maps $\Hc_a$ onto $\Hc_b$. 
Nevertheless, since $\Hc_a$ and $\Hc_b$ are closed infinite-dimensional subspaces of the separable Hilbert space $\Hc$, there exists a partial isometry $v\in\Bc(\Hc)$ whose initial subspace is $\Hc_a$ and whose final subspace is~$\Hc_b$. 
If we define $b:=vav^*$, it follows that $\Hc_b$ is the closure of the range of $b$. 
If $b=uau^*$ for some unitary operator $u\in\U(\Hc)$, then $u(\Hc_a)=\Hc_b$, which is impossible by the above remarks. 
Thus $b\not\in\Oc_{\Bc(\Hc)}(a)$. 
However $b\in\overline{\Oc_{\Bc(\Hc)}(a)}^{\Vert\cdot\Vert}$ by Lemma~\ref{grpd}. 
Moreover $\overline{\Oc_{\Bc(\Hc)}(a)}^{\Vert\cdot\Vert_{\Ic}}
=\overline{\Oc_{\Bc(\Hc)}(a)}^{\Vert\cdot\Vert}$ by Theorem~\ref{closures_th}, 
hence  $b\in\overline{\Oc_{\Bc(\Hc)}(a)}^{\Vert\cdot\Vert_{\Ic}}$. 
This shows that $\Oc_{\Bc(\Hc)}(a)$ is not a closed subset of $\Ic$, and we are done. 
\end{proof}

\begin{remark}
Corollary~\ref{closed} for $\Ic=\Kc(\Hc)$ shows that for any $a\in\Kc(\Hc)^\nor$, its unitary orbit $\Oc_{\Bc(\Hc)}(a)$ is closed in $\Bc(\Hc)$ if and only if $a$ has finite spectrum, or, equivalently, finite rank. 
\end{remark}


\section{On the orbits of the action groupoid $\Vc(\M)\ast \M^\nor\tto \M^\nor$
}
\label{sec:groupoid}

In the present section we use some basic terminology related to groupoids. 
A quite convenient reference in this connection is \cite[App.]{OS}, while the general theory of Banach-Lie groupoids is developed in \cite{BGJP19}. 

Motivated by the results of Section~\ref{sec:closure} (e.g., Lemma~\ref{grpd}), we investigate the action of the groupoid of partial isometries on the set of normal operators in a von Neumann algebra. 
Specifically, we establish smoothness properties of the corresponding groupoid orbits (Corollary~\ref{smcl0}) and we describe the norm closures of the groupoid orbits (Proposition~\ref{P3} and Corollary~\ref{C4}). 
Finally, we study the natural topologies carried by the groupoid orbits of normal operators (Theorems \ref{immersed} and \ref{top_th}). 
We also obtain some results regarding groupoid orbits in von Neumann algebras, inspired by \cite[Th. 3.4]{BO19} on Banach manifold structures  
of the coadjoint orbits of the Banach-Lie groupoid of partial isometries. 
However, the present approach is completely different 
and, even in the special case $\M=\Bc(\Hc)$, the present result can be recovered from \cite[Th. 3.4]{BO19} 
only in the special case of self-adjoint trace-class operators, 
while here we study arbitrary \emph{normal} operators. 

The following definition is an illustration of the general framework of groupoid actions; see for instance \cite[App. C]{OS}.


\begin{definition}
	Let $\M$ be any $W^*$-algebra with its lattice of projections 
$$
\Pc(\M):=\{p\in\M\mid p=p^2=p^*\},
$$
and its set of partial isometries 
$$\Vc(\M):=\{v\in\M\mid v^*v\in\Pc(\M)\}
$$
regarded as a Banach-Lie groupoid 
$\Vc(\M)\tto\Pc(\M)$. 
(See for instance \cite{BGJP19} and the references therein). 
We recall the notation
$$
\M^\sa:=\{x\in\M\mid x^*=x\}\textrm{ and }\M^\nor:=\{y\in\M\mid y^*y=yy^*\}
$$ 
for the sets of self-adjoint and of normal operators respectively. 
For every $y\in \M^\nor$ let $\bp(y):=p_y$. 
This is the orthogonal projection onto $\overline{\Ran\, y}$  for any faithful representation $\M=\M''\subseteq\Bc(\Hc)$, and it can be abstractly defined e.g., as the smallest $p\in\Pc(\M)$ satisfying $py=y$.

We then define $\Vc(\M)\ast \M^\nor\tto \M^\nor$ as the \textit{action groupoid} 
given by the action of the groupoid of partial isometries $\Vc(\M)\tto\Pc(\M)$ 
on $\M^\nor$ via the moment map $\bp\colon \M^\nor\to \Pc(\M)$. 
More specifically, 
$$\Vc(\M)\ast \M^\nor:=\{(v,y)\in \Vc(\M)\times \M^\nor \mid v^*v=\bp(y)\}$$
the source/target maps are 
$$
\xymatrix{\M^\nor & \Vc(\M)\ast \M^\nor \ar[l]_{\bt\quad} \ar[r]^{\quad\bs} & \M^\nor},
\ \bs(v,y):=y,\   
 \bt(v,y):=vyv^*
 $$ 
while the multiplication is given by $$(v_1,y_1)\cdot(v_2,y_2):=(v_1v_2,y_2)$$ 
if $(v_1,y_1),(v_2,y_2)\in \Vc(\M)\ast \M^\nor$ satisfy $\bs(v_1,y_1)=\bt(v_2,y_2)$. In particular $v_1^*v_1=v_2v_2^*$, so $v_1v_2\in\Vc(\M)$. For later use, we note the following: 
\begin{equation}\label{6Apr2019_var}
\text{ if }(v,y)\in \Vc(\M)\times \M^\nor\text{ then }\bp(vyv^*)=v\bp(y)v^*.
\end{equation}
The \emph{groupoid orbit} of any $a\in\M^\nor$ is 
$$
\Oc_a:=\{vav^*\mid (v,a)\in \Vc(\M)\ast \M^\nor\}.
$$ 
\end{definition}

\subsection{On the closures of groupoid orbits}

\begin{lemma}\label{L1}
	For every $W^*$-algebra the following assertions hold. 
	\begin{enumerate}[{\rm(i)}]
		\item\label{L1_item1}
		If $a\in\M^\nor$ and $b\in\overline{\Oc_a}$ then $\spec(a)\cup\{0\}=\spec(b)\cup\{0\}$. 
		\item\label{L1_item2} 
		If $a\in\M^\nor$, $b\in\overline{\Oc_a}$, and $\{a_n\}_{n\ge 1}$ is a sequence in $\Oc_a$ with $\lim\limits_{n\to\infty}a_n=b$ in the operator norm topology, 
		then for every function $f\in\Cc(\spec(a)\cup\{0\})$ we have  $\lim\limits_{n\to\infty}f(a_n)=f(b)$ in the operator norm topology. 
	\end{enumerate}
\end{lemma}

\begin{proof}
	\eqref{L1_item1}  
	Since $b\in\overline{\Oc_a}$ there exists a sequence $\{a_n\}_{n\ge 1}$  in $\Oc_a$ with $\lim\limits_{n\to\infty}a_n=b$ in the operator norm topology. 
	For every $n\ge 1$ we have $a_n\in\Oc_a$ and then it is easily checked that $\spec(a_n)\cup\{0\}=\spec(a)\cup\{0\}$. 
	On the other hand, we also have $a_n\in\M^\nor$. 
	The assertion then follows by the norm continuity of the spectrum on the set of normal operators. 
	(See for instance \cite[Sol. 105]{Ha82}.)
	
	\eqref{L1_item2} 
	We have already seen that  $\spec(b)\cup\{0\}=\spec(a)\cup\{0\}=\spec(a_n)\cup\{0\}$ for all $n\ge1$, 
	hence the assertion follows by the norm continuity of the functional calculus with continuous functions. 
	(See for instance \cite[Sol. 126]{Ha82}.)
\end{proof}

\begin{lemma}\label{L2}
	Let $a,b\in\Bc(\Hc)^\nor$ with their spectral measures denoted by $E^a$ and $E^b$, respectively. 
	Then for arbitrary $\varepsilon\in(0,\infty)$ there exist an integer $r\ge 1$, mutually disjoint open sets $\Delta_1,\dots,\Delta_r\subseteq\CC$, and a closed subset $N\subseteq\CC\setminus(\Delta_1\sqcup\cdots\sqcup\Delta_k)$ satisfying 
	\begin{equation}\label{L2_eq1}
	\spec(a)=N\sqcup\bigsqcup_{k=1}^r(\spec(a)\cap\Delta_k)
	\end{equation}
	and moreover $\vert z-w\vert\le\varepsilon $ if $z,w\in\Delta_k$ for $k=1,\dots,r$, and $E^a(N)=E^b(N)=0$. 
\end{lemma}

\begin{proof}
	This is a by-product of the proof of \cite[Th. 1]{GePa74}.
\end{proof}

In the following statement we use the traditional notation $\sim$ for Murray-von Neumann equivalence of projections.  
That is, if $\M$ is a $W^*$-algebra and $p,q\in\Pc(\M)$, then we write $p\sim q$ if and only if there exists $v\in\M$ with $v^*v=p$ and $vv^*=q$. 
In terms of groupoid orbits we have $\Oc_p=\{q\in\Pc(\M)\mid p\sim q\}$ for every $p\in\Pc(\M)$, 
as noted in \cite{OS}. 

\begin{proposition}\label{P3}
	Consider the von Neumann algebra $\M=\Bc(\Hc)$. 
	If $a,b\in\M^\nor$ with their spectral measures denoted by $E^a$ and $E^b$, respectively, then the following conditions are equivalent: 
	\begin{enumerate}[{\rm(i)}]
		\item\label{P3_item1}
		One has $b\in\overline{\Oc_a}$. 
		\item\label{P3_item2} 
		One has $\spec(a)\cup\{0\}=\spec(b)\cup\{0\}$ and $E^a(\Delta)\sim E^b(\Delta)$ in $\M$ for every open subset $\Delta\subseteq\CC\setminus\{0\}$.  
	\end{enumerate}
\end{proposition}

\begin{proof}
	``\eqref{P3_item1}$\implies$\eqref{P3_item2}'' 
	We have $\spec(a)\cup\{0\}=\spec(b)\cup\{0\}$ by Lemma~\ref{L1}\eqref{L1_item1}. 
	For an arbitrary open subset $\Delta\subseteq\CC\setminus\{0\}$ we now separately discuss two cases that may occur: 
	
	Case 1: The set $\Delta\cap\spec(a)$ is infinite, hence $\Delta\cap\spec(b)$ is infinite as well. 
	Then it is easily checked that both projections $E^a(\Delta),E^b(\Delta)\in\Pc(\M)$ have infinite rank, hence $E^a(\Delta)\sim E^b(\Delta)$ in $\M=\Bc(\Hc)$. 
	
	Case 2: The set $\Delta\cap\spec(a)$ is finite, 
	hence $\Delta\cap\spec(b)$ is finite as well. 
	Then the characteristic function $f:=\chi_{\Delta\cap\spec(a)}$  satisfies $f\in\Cc(\spec(a)\cup\{0\})$. 
	For any sequence $\{a_n\}_{n\ge 1}$  in $\Oc_a$ with $\lim\limits_{n\to\infty}a_n=b$ in the operator norm topology 
	we then have $\lim\limits_{n\to\infty}f(a_n)=f(b)$ in the operator norm topology again, 
	by Lemma~\ref{L1}\eqref{L1_item1}. 
	Here $f(b)=E^b(\Delta)$ and $f(a_n)=E^{a_n}(\Delta)$ are orthogonal projections for every $n\ge 1$ 
	by the way the function $f$ was chosen, hence there exists $n_0\ge 1$ for which the rank of $f(a_n)$ is equal to the rank of $f(b)$ for every $n\ge n_0$. 
	On the other hand, for every $n\ge 1$ we have $a_n\in\Oc_a$, and then it is easily checked that the rank of $f(a_n)$ is equal to the rank of $f(a)$ 
	since $\Delta\subseteq\CC\setminus\{0\}$. 
	Consequently the rank of $f(a)$ is equal to the rank of $f(b)$, 
	that is, $E^a(\Delta)\sim E^b(\Delta)$ in $\M=\Bc(\Hc)$.

	``\eqref{P3_item2}$\implies$\eqref{P3_item1}'' 
	We show that for arbitrary $\varepsilon\in(0,\infty)$ there exists $v\in\M$ with $v^*v=\bp(a)$ and $\Vert vav^*-b\Vert\le2\varepsilon$. 
	
	To this end we first use Lemma~\ref{L2} to find the partition~\eqref{L2_eq1}. 
	Then 
	$$\spec(b)\setminus\{0\}=
	\spec(a)\setminus\{0\}
	=(N\setminus\{0\})\sqcup\bigsqcup_{r=1}^k((\spec(a)\setminus\{0\})\cap\Delta_k)$$
	with $E^a(N\setminus\{0\})\le E^a(N)=0$ and 
	$E^b(N\setminus\{0\})\le E^b(N)=0$. 
	Therefore we obtain the sums of mutually orthogonal projections in $\M$
	\begin{align}
	\label{P3_proof_eq1}
	\1 &=E^a(\{0\})+\sum_{k=1}^r E^a((\spec(a)\setminus\{0\})\cap\Delta_k), \\
	\label{P3_proof_eq2}
	\1 &=E^b(\{0\})+\sum_{k=1}^r E^b((\spec(b)\setminus\{0\})\cap\Delta_k).  
	\end{align}
	For $k=1,\dots,r$ let us denote $p_k:=E^a((\spec(a)\setminus\{0\})\cap\Delta_k)$ and $q_k:=E^b((\spec(b)\setminus\{0\})\cap\Delta_k)$. 
	We have $p_k\sim q_k$ in $\M$ by the hypothesis \eqref{P3_item2}, 
	hence there exists $v_k\in\M$ with $v_k^*v_k=p_k$ and $v_kv_k^*=q_k$. 
	Defining 
	$$v:=v_1+\cdots+v_r\in\M$$
	we have 
	$v^*v=p_1+\cdots+p_r=\bp(a)$ by \eqref{P3_proof_eq1} 
	and similarly $vv^*=\bp(b)$  by \eqref{P3_proof_eq2}.
	
	We now check that $\Vert vav^*-b\Vert\le2\varepsilon$. 
	To this end select $z_k\in\Delta_k$ arbitrary. 
	Then $\vert w-z_k\vert\le\varepsilon$ for all $w\in\Delta_k$. 
	(See Lemma~\ref{L2}.) 
	It is then straightforward to check that 
	$$\Vert (b-z_k\1)q_k\Vert\le\varepsilon\text{ and }\Vert (a-z_k\1)p_k\Vert\le\varepsilon.$$ 
	On the other hand, 
	$q_kv=q_kv_k=v_kv_k^*v_k=v_kp_k$ hence 
	$v^*q_k=p_kv_k^*=p_kv_k^*q_k$, and then 
	$$vav^*q_k
	=vap_kv_k^*q_k
	=v(a-z_k\1)p_kv^*q_k+z_k vp_kv^*q_k
	=v(a-z_k\1)p_kv^*q_k+z_k q_k.$$
	This further implies 
	$$(vav^*-b)q_k=v(a-z_k\1)p_kv^*q_k+(z_k-b)q_k$$
	hence 
	$$\Vert (vav^*-b)q_k\Vert\le \Vert (a-z_k\1)p_k\Vert+\Vert (b-z_k\1)q_k\Vert
	\le2\varepsilon
	\text{ for }k=1,\dots,r.$$
	Moreover, the projection $q_k$ commutes with both operators $vav^*$ and $b$ for $k=1,\dots,r$, hence 
	$$\Vert (vav^*-b)(q_1+\cdots+q_r)\Vert
	\le2\varepsilon.$$
	Finally, since $\bp(vav^*)=v\bp(a)v^*=vv^*vv^*=vv^*=q_1+\cdots+q_r=\bp(b)$, 
	we actually have $\Vert vav^*-b\Vert\le2\varepsilon$, and this completes the proof. 
\end{proof}

\begin{corollary}\label{C4}
	Consider the von Neumann algebra $\M=\Bc(\Hc)$. 
	If $a\in\M^\nor$, then the following conditions are equivalent: 
	\begin{enumerate}[{\rm(i)}]
		\item\label{C4_item1}
		The groupoid orbit of the operator $a$ is closed in $\M$ in the operator norm topology. 
		\item\label{C4_item2} 
		The compact set $\spec(a)$ accumulates at most at the point $0\in\CC$.  
	\end{enumerate}
\end{corollary}

\begin{proof}
	\eqref{C4_item1}$\implies$\eqref{C4_item2} 
	Assume there exists an accumulation point $z_0\in\spec(a)\setminus\{0\}$ 
	and denote $p_0:=E^a(\{z_0\})\in\Pc(\M)$. 
	Two cases may occur: 
	
	Case 1: $p_0\ne0$, that is, $z_0$ is an eigenvalue of the normal operator~$a$. 
	If we define 
	$$b:=(\1-p_0)a
	=a-z_0p_0
	\in\M$$ 
	then it is easily seen that $\spec(b)=\spec(a)$ since $z_0$ is an accumulation point of $\spec(a)$. 
	Moreover 
	the projections $E^a(\Delta)$ and $E^b(\Delta)$ have equal ranks 
	for every open set $\Delta\subseteq\CC\setminus\{0\}$. 
	Indeed, if $z_0\in\Delta$, 
	then both $E^a(\Delta)$ and $E^b(\Delta)$ have infinite rank, 
	while if $z_0\not\in\Delta$, 
	when $E^a(\Delta)=E^b(\Delta)$. 
	Therefore $b\in\overline{\Oc_a}$ by Proposition~\ref{P3}(\eqref{P3_item2}$\implies$\eqref{P3_item1}). 
	On the other hand, it is easily seen that $z_0$ is not an eigenvalue of $b$, hence $E^b(\{z_0\})=0\ne E^a(\{z_0\})$, 
	which directly implies that $b\not\in\Oc_a$, since $z_0\ne0$. 
	
	Case 2: $p_0=0$. 
	We then write $\Hc=\Hc_1\oplus\Hc_2$ where both summands are infinite-dimensional. 
	For any unitary operator $v\colon \Hc\to\Hc_1$, if we define 
	$$b:=vav^*\oplus z_0\id_{\Hc_2}\in\Bc(\Hc)=\M$$ 
	then we have $\spec(b)=\spec(a)$ and 
	the projections $E^a(\Delta)$ and $E^b(\Delta)$ have equal ranks 
	for every open set $\Delta\subseteq\CC\setminus\{0\}$. 
	Therefore $b\in\overline{\Oc_a}$ by Proposition~\ref{P3}(\eqref{P3_item2}$\implies$\eqref{P3_item1}). 
	On the other hand $E^b(\{z_0\})\ne0 = E^a(\{z_0\})$, 
	which again implies that $b\not\in\Oc_a$ since $z_0\ne0$. 
	
	\eqref{C4_item2}$\implies$\eqref{C4_item1} 
	For arbitrary $b\in\overline{\Oc_a}$ we have $\spec(b)\setminus\{0\}=\spec(a)\setminus\{0\}$ by Lemma~\ref{L1}, 
	hence the hypothesis on $\spec(a)$ implies that $\spec(b)\setminus\{0\}$ is a set of isolated points of $\spec(b)$. 
	It then easily follows by Proposition~\ref{P3}(\eqref{P3_item1}$\implies$\eqref{P3_item2}) that 
	for every  for every $\lambda\in\spec(a)\setminus\{0\}$ 	
	there exists $v_\lambda\in\M$ with $v_\lambda^*v_\lambda=E^a(\{\lambda\})$ and 
	$v_\lambda v_\lambda^*=E^b(\{\lambda\})$. 
	Defining 
	$$v:=\sum_{\lambda\in\spec(a)\setminus\{0\}}v_\lambda\in\M$$
	we then obtain $v^*v=\bp(a)$ and $vv^*=\bp(b)$. 
	On the other hand, since $a$ and $b$ are normal operators, it follows that $E^a(\{\lambda\})$ and $E^b(\{\lambda\})$ are the orthogonal projections onto the eigenspaces of $a$ and $b$ corresponding to the eigenvalue $\lambda$, respectively. 
	This easily implies $vav^*=b$, hence $b\in\Oc_a$. 
	Since this holds for  arbitrary $b\in\overline{\Oc_a}$, it follows that $\Oc_a$ is closed in $\M$ with respect to the operator norm topology. 
\end{proof}

\begin{remark}\label{R5}
	\normalfont 
	In Proposition~\ref{P3}, the implication \eqref{P3_item2}$\implies$\eqref{P3_item1} actually holds true for any $W^*$-algebra $\M$ whose predual is separable, since such an $\M$ can be realized as a von Neumann algebra on a separable Hilbert space, hence Lemma~\ref{L2} is applicable for the normal operators in~$\M$. 
	The method of proof of Proposition~\ref{P3} is inspired by the proof of \cite[Th. 1]{GePa74}.
\end{remark}

\subsection{Smoothness properties of the groupoid orbits}

\begin{lemma}\label{BR05_Th2.2}
If $\M$ is a $W^*$-algebra,   
$a\in\M^\nor$ and $\M_a:=\{x\in\M\mid xa=ax\}$, 
then $\M_a$ is a $W^*$-subalgebra of $\M$, there exists a conditional expectation $\widetilde{E}\colon \M\to\M_a$, 
and  the unitary group $U(\M_a)$ 
is a Banach-Lie subgroup of $U(\M)$.
\end{lemma}

\begin{proof}
Since $a\in\M$ and $aa^*=a^*a$, there exist $a_1,a_2\in\M^\sa$ with $a_1a_2=a_2a_1$ and $a=a_1+\ie a_2$. 
Then
$\M_a=\{x\in\M\mid xa_j=a_jx\text{ for }j=1,2\}$ by the Fuglede theorem, 
and this directly implies that $\M_a$ is a $W^*$-subalgebra of $\M$. 
Furthermore, let us consider the abelian (semi)-group $S=(\RR^2,+)$ 
and define 
$$\alpha\colon S\to\Bc(\M),\quad 
\alpha^{(t_1,t_2)}(x):=\ee^{\ie(t_1a_1+t_2a_2)}x\ee^{-\ie(t_1a_1+t_2a_2)}.$$
Then it is easily seen that Lemma~\ref{averaging} is applicable 
with $\Xc=\M$ and $\Xc^S=\M_a$. 
We thus obtain $\widetilde{E}\colon \M\to\M$ with $\widetilde{E}^2=\widetilde{E}$, $\Vert \widetilde{E}\Vert\le 1$, 
and $\Ran \widetilde{E}=\M_a$. 
It follows by Tomiyama's classical theorem that $\widetilde{E}$ is a conditional expectation, in particular $\widetilde{E}(x^*)=\widetilde{E}(x)^*$ for all $x\in\M$. 
This implies $\widetilde{E}(\ug(\M))\subseteq \ug(\widetilde{E})$, 
hence $\widetilde{E}(\ug(\M))=\ug(\M_a)$, $\ug(\M)=\ug(\M_a)\dotplus\Ker(\widetilde{E}\vert_{\ug(\M)})$, 
 and one obtains that 
$U(\M_a)$  is a Banach-Lie subgroup of~$U(\M)$ just as in the proof of Proposition~\ref{manifold_bdd}. 
\end{proof}

\begin{lemma}\label{BGJP19_Th.3.3}
Let $\mathcal{G}\tto M$ be a split Banach-Lie groupoid with its source/target maps $\bs,\bt\colon\Gc\to M$. 
If  $m_0,m_1\in M$,  
then $\mathcal{G}(m_0,m_1):=\bs^{-1}(m_0)\cap \bt^{-1}(m_1)$ is an embedded submanifold of $\mathcal{G}$, 
the isotropy group $\mathcal{G}(m_0):=\bs^{-1}(m_0)\cap \bt^{-1}(m_0)$ is a Banach-Lie group, 
and for every $h_0\in  \mathcal{G}(m_0,m_1)$ the mapping 
$\mathcal{G}(m_0)\to\mathcal{G}(m_0,m_1)$, $g\mapsto h_0g$ is a split immersion. 
\end{lemma}

\begin{proof}
This is a by-product of the proof of \cite[Th. 3.3]{BGJP19}.
\end{proof}

In the proof of the following results we use ideas from the proof of \cite[Th. 3.3(iii)]{BGJP19} on smooth orbits of some Banach-Lie groupoids.  We emphasize however that $\Vc(\M)\ast \M^\sa\tto \M^\sa$ is {\it not} a Banach-Lie groupoid, since for instance its space of arrows $\Vc(\M)\ast \M^\sa$ is not a Banach manifold.

\begin{definition}
\label{princ}
Let $\M$ be a $W^*$-algebra. 
Consider the isotropy group of $a$ with respect to the groupoid action  $\Vc(\M)\ast \M^\nor\tto \M^\nor$, which is given by
$$K_a:=\{w\in\Vc(\M)\mid w^*w=p_a,\ waw^*=a\}$$
with identity element $p_a\in K_a$. If we let
$$\Vc_a=\{v\in\Vc(\M)\mid v^*v=p_a\}$$
we have $K_a\subseteq \Vc_a$ and one has a natural right action of the group $K_a$ on $\Vc_a$ by 
\begin{equation}
\label{smooth_proof_eq1}
\Vc_a\times K_a\to \Vc_a,\quad (v,w)\mapsto vw.
\end{equation} 
\end{definition}

We emphasize that the group action \eqref{smooth_proof_eq1} is \emph{not} transitive.

\begin{remark}\label{quotient1}
In Definition~\ref{princ}, 
 the set $\Vc_a$ is a closed embedded Banach submanifold of $\M$,  since it is a fiber of the mapping $\Vc(\M)\to\Pc(\M)$, $v\mapsto v^*v$, a submersion of Banach manifolds by \cite[Prop. 3.4]{AnCoMb05}. 
 The construction of a specific smooth atlas of $\Vc_a$ can be found in  \cite[Subsect. 2.5]{BO19}, 
 where the set $P_0$ coincides with the present $\Vc_a$ if $p_0=p_a$.

Note that for every $w\in K_a$ one has $ww^*=p_a$ by~\eqref{6Apr2019_var}. Thus $K_a$ is a subgroup of the unitary group of the $W^*$-algebra $p_a\M p_a$. 
Using Lemma~\ref{BR05_Th2.2} we then obtain that $K_a$ is a Banach-Lie subgroup of $U_a=U(p_a\M p_a)$. 

Now note that the mapping $\Psi\colon \Vc_a/K_a\to \Oc_a$, $\Psi(vK_a):=vav^*$, is a well-defined bijection. Therefore to give the orbit a differentiable structure it suffices to obtain one for the quotient space.
\end{remark}

\begin{proposition}\label{smooth}
Let $\M$ be a $W^*$-algebra. 
If $a\in\M^\nor$, then its groupoid orbit $\Oc_a$ has the structure of a Banach manifold 
for which the inclusion map $\Oc_a\hookrightarrow\M$ is smooth 
and its tangent map at every point of the orbit is injective. 
Moreover the action map $\alpha_a\colon\Vc_a\to \Oc_a$, $v\mapsto vav^*$, is a smooth principal bundle whose structural group is~$K_a$. 
\end{proposition}

\begin{proof}
To obtain the Banach manifold structure of $\Oc_a\simeq \Vc_a/K_a$ for which the  quotient map $\Vc_a\to \Vc_a/K_a$ is a submersion, we can use \cite[Lemma 2.5]{BGJP19}.  To this end we must check that the following conditions are satisfied: 
\begin{enumerate}[{\rm(a)}]
	\item\label{smooth_proof_item1} The group action \eqref{smooth_proof_eq1} is smooth, free, and proper. 
	\item\label{smooth_proof_item2} For every $v\in \Vc_a$ the mapping $L_v\colon K_a\to \Vc_a$, $w\mapsto vw$, is a split immersion. 
\end{enumerate}

\eqref{smooth_proof_item1} 
The group action \eqref{smooth_proof_eq1} is smooth since 
it is the restriction of a bilinear map, which is smooth. 
To check that the group action \eqref{smooth_proof_eq1} is proper, one may use \cite[Lemma 2.3]{BGJP19}.
Finally, the group action \eqref{smooth_proof_eq1}  is free since it is the restriction of a groupoid multiplication. 

\eqref{smooth_proof_item2} 
Recall that $U_a=U(p_a\M p_a)$, and note that $U_a$ is the isotropy group of $p\in\Pc(\M)$ 
in the Banach-Lie groupoid $\Vc(\M)\tto\Pc(\M)$.  Then  the right group action \eqref{smooth_proof_eq1} extends to the smooth action 
$$\Vc_a\times U_a\to \Vc_a,\quad (v,u)\mapsto vu,$$
which leads to the smooth mapping $\widetilde{L}_v\colon U_a\to \Vc_a$, $u\mapsto vu$. 
It follows by Lemma~\ref{BGJP19_Th.3.3} that $\widetilde{L}_v$ is a split immersion. 
Since we noted in the remark before this proposition that $K_a$ is a Banach-Lie subgroup of $U_a$ 
and $L_v=\widetilde{L}_v\vert_{K_a}$, it then follows that the mapping $L_v\colon K_a\to \Vc_a$ is a split immersion, as claimed.

It remains to show that the inclusion map  $\iota\colon \Oc_a\hookrightarrow\M$ is smooth 
and its tangent map at every point is injective.  
To this end we use the commutative diagram 
$$\xymatrix{
\Vc_a \ar[d]_{q} \ar[r]^{\alpha_a} & \M \\
\Vc_a/K_a \ar[r]^{\Psi} & \Oc_a \ar@{^{(}->}[u]_{\iota}
}$$
in which $q\colon \Vc_a\to \Vc_a/K_a$ is the quotient map and $\alpha_a\colon \Vc_a\to\M$, $\alpha_a(v):=va v^*$. Here $q$ is a submersion and $\Psi$ is by definition a diffeomorphism. 
Therefore, in order to complete the proof, it suffices to show that $\alpha$ is smooth and for every $v\in \Vc_a$ one has 
\begin{equation*}
\Ker\, (T_v(\alpha_a))= \Ker\, (T_v q)
\end{equation*}
which is clear. 
\end{proof}

\begin{example}
\label{grass}
\normalfont 
As mentioned at the beginning of this section, we work in the setting of Banach-Lie groupoids as developed in \cite{BGJP19}. 
In particular, the model spaces of the connected components of the Banach manifolds need not be isomorphic to each other. 
To illustrate this point, we specialize Proposition~\ref{smooth} to the case when $a\in\M^\nor$ is an injective operator, or, equivalently, $p_a=\1\in\M$. 
Then 
$$\Vc_a=\{v\in\M\mid v^*v=\1\}$$
is the set of all isometries in~$\M$. 
 
If $\M$ is a finite $W^*$-algebra, then $\Vc_a=\U(\M)$, hence 
$$\Oc_a=\{vav^*\mid v\in\U(\M)\}\simeq \U(\M)/\U(\M_a)$$
and $K_a=\U(\M_a)$, 
where $\M_a$ is defined in Lemma~\ref{BR05_Th2.2}. 
Thus $\Oc_a$ is a connected Banach manifold in this case, since $\U(\M)$ is a connected Banach-Lie group. 

If however $\M=\Bc(\Hc)$ for a separable infinite-dimensional Hilbert space~$\Hc$, 
let us further assume $a=\1$. 
Then 
\begin{align*}
\Oc_a
&=\{vav^*\mid v\in\Vc_a\}=\{vv^*\mid v\in\Bc(\Hc),\ v^*v=\1\} \\
&=\{p\in\Bc(\Hc)\mid p=p^2=p^*,\ \dim (p\Hc)=\infty\}.
\end{align*}
and this has the structure of a Banach manifold by Proposition~\ref{smooth}. 
This is actually the union of some of the connected components of the Grassmann manifold $\Pc(\M)$ for $\M=\Bc(\Hc)$. 
Specifically, for any $n\ge 1$, let us fix a projection $p_n\in\Bc(\Hc)$ with $\dim(p_n\Hc)=n$, 
and let us also fix a projection $p_\infty\in\Bc(\Hc)$ with $\dim(p_\infty\Hc)=\dim((\1-p_\infty)\Hc)=\infty$. 
Then we have the disjoint union of connected components 
$$\Oc_a=\Oc(p_\infty)\sqcup\bigsqcup_{n\ge 1}\Oc(\1-p_n).$$
We note that for any projection $p\in\Bc(\Hc)$ its unitary orbit $\Oc(p)$ 
is diffeomorphic tothe smooth homogeneous space $\U(\Hc)/\U(\{p\}')$ and we have an isomorphism of Banach-Lie groups $\U(\{p\}')\simeq \U(p\Hc)\times\U((\1-p)\Hc)$. 
\end{example}

\begin{definition}[two topologies]\label{sometops} 
Let $\M$ be a $W^*$-algebra, $a\in\M^\nor$, and 
	recall the bijection $\Psi\colon \Vc_a/K_a\to \Oc_a$ from Remark~\ref{quotient1}. 
The manifold topology of the groupoid orbit~$\Oc_a$ is the quotient topology for the open map $\alpha_a:\Vc_a\to  \Oc_a$, $\alpha_a(v):=vav^*$, hence we will denote the corresponding topological space by $(\Oc_a,\tau_q)$. 
Since $\Vc_a$ has the norm topology, a basis of neighborhoods for this topology at the point $q(p_a)=a\in \Oc_a$ is given by
$$
W_{\varepsilon}:=
\{\alpha_a(w): w\in\Vc_a, 
\, \|w-p_a\|<\varepsilon\}
=\{w aw^*: 
w^*w=p_a, 
\, \|w-p_a\|<\varepsilon\}.
$$
For given $v\in \Vc_a$ and $b=\alpha_a(v)=v a v^*\in \Oc_a$, a basis of neighborhoods around $b$ is given by the sets
$$Z_\varepsilon:=
\{\alpha_a(z): z\in\Vc_a,\, \|z-v \|<\varepsilon\}
=
\{z az^* : z^*z=p_a,\, 
\|z-v \|<\varepsilon\}.
$$
The norm topology of $\Oc_a$ will be denoted by $(\Oc_a,\tau_{\M})$. 
It is clear that the identity map $i:(\Oc_a,\tau_q)\to (\Oc_a,\tau_{\M})$ is continuous, for, if $w_n\in\Vc_a$ and $\lim\limits_{n\to\infty}\Vert w_n-p_a\Vert=0$, then $\lim\limits_{n\to\infty}\Vert w_n a w_n^*-a\Vert= 0$.
\end{definition}

The topology of $\Oc_a$ considered in the following corollary is the manifold topology, i.e., the quotient topology described above:

\begin{corollary}\label{smcl0}
	Let $\M$ be a $W^*$-algebra. 
	Let $a\in\M^\nor$,  let $p_a$ be its range projection. Let 
	$$\Vc_a=\{v\in\Vc(\M)\mid v^*v=p_a\}\textrm{ and }K_a:=\{v\in \Vc_a\mid vav^*=a\}, 
	$$ 
	and  let $\Oc(a)=\{uau^*\mid u\in U(\M)\}$ be the unitary orbit of~$a$.  Then $\Vc_a$ is a submanifold of $\M$, $K_a$ is a Banach-Lie subgroup and a submanifold of $\Vc_a$, and 
	the sets $\Oc(a)\subseteq \Oc_a$ have structures of Banach manifolds having the following properties: 
	\begin{enumerate}[{\rm(i)}]
		\item\label{smcl0_item1} The mapping $\alpha_a\colon \colon \Vc_a\to \Oc_a$, $v\mapsto vav^*$, is a principal bundle whose structural group is $K_a$. 
		\item\label{smcl0_item2} The inclusion map $\Oc_a\hookrightarrow \M$ is smooth and its tangent map at every point is injective. 
		\item\label{smcl0_item3} The inclusion map $\iota_a\colon \Oc(a)\hookrightarrow\Oc_a$ is smooth and its tangent map at every point is injective. 
	\end{enumerate}
\end{corollary}

\begin{proof} 	
\eqref{smcl0_item1}--\eqref{smcl0_item2} 
These assertions follow by Proposition~\ref{smooth}.  

\eqref{smcl0_item3} 
The mapping $\kappa\colon U(\M)\to \Vc_a$, $\kappa(u):=up$,  is smooth since~$\Vc_a$ is a submanifold of $\M$. 
Moreover, one has $\Oc(a)=q_a(\kappa(U(\M)))\subseteq q_a(\Vc_a)=\Oc_a$ 
and the diagram 
$$\xymatrix{U(\M) \ar[d]_{q_0} \ar[r]^{\kappa}& \Vc_a \ar[d]^{q_a}\\
\Oc(a) \ar@{^{(}->}[r]^{\iota_a}&  \Oc_a}$$
is commutative and its vertical arrows are submersions, where $q_0(u):=uau^*$ for all $u\in U(\M)$. 
Since $\kappa$ is smooth, it then follows that $\iota_a$ is smooth. 
Moreover, since the inclusion map $\tilde{\iota}_a\colon \Oc_a\hookrightarrow\M$ is smooth and its tangent map is injective by \eqref{smcl0_item2}, while $\tilde{\iota}_a\circ\iota_a$ is the inclusion map  $\Oc(a)\hookrightarrow \M$ whose tangent map is well known to be injective at every point of $\Oc(a)$, it then follows that the tangent map of $\iota_a$ is injective at every point of $\Oc(a)$. 
\end{proof}

Recall that $\Vc(\Hc)$ denotes the set of partial isometries in $\Bc(\Hc)$.  Then, combining Corollary \ref{smcl0} with Lemma \ref{grpd}, we obtain the following fact. 

\begin{corollary}\label{smcl}
Let $a\in\Kc(\Hc)^\nor$ 
and  
$$\Vc_a:=\{v\in\Vc(\Hc)\mid v^*v=p_a\}\textrm{ and }K_a:=\{v\in \Vc_a\mid vav^*=a\}.
$$ 
Then $\Vc_a$ is a submanifold of $\Bc(\Hc)$, $K_a$ is a Banach-Lie group and a submanifold of $\Vc_a$, and 
the set $\overline{\Oc_{\Bc(\Hc)}(a)}^{\Vert\cdot\Vert}$ has the structure of a Banach manifold having the following properties: 
\begin{enumerate}[{\rm(i)}]
\item The mapping $\Vc_a\to \overline{\Oc_{\Bc(\Hc)}(a)}^{\Vert\cdot\Vert}$,  $v\mapsto vav^*$, is a principal bundle whose structural group is $K_a$. 
\item The inclusion map $\overline{\Oc_{\Bc(\Hc)}(a)}^{\Vert\cdot\Vert}\hookrightarrow \Kc(\Hc)$ is smooth and its tangent map at every point is injective. 
\item The inclusion map $\Oc_{\Bc(\Hc)}(a)\hookrightarrow \overline{\Oc_{\Bc(\Hc)}(a)}^{\Vert\cdot\Vert}=\Oc_a$ 
is smooth and its tangent map at every point is injective. 
\end{enumerate}
\end{corollary}

\begin{proof}
It follows by Lemma~\ref{grpd} that the mapping $\Vc_a\to \overline{\Oc_{\Bc(\Hc)}(a)}^{\Vert\cdot\Vert}$, $v\mapsto vav^*$, is surjective, hence $\overline{\Oc_{\Bc(\Hc)}(a)}^{\Vert\cdot\Vert}=\Oc_a$. 
The other assertions then follow by Corollary~\ref{smcl0}. 
\end{proof}

\begin{remark}[weak immersions versus immersions] 
As mentioned in Corollary~\ref{smcl0}, 
if $\M$ is a $W^*$-algebra and $a\in\M^\nor$, 
then the inclusion $i:\Oc_a\to \M$ is smooth, and its differential $i_{*}\colon T\Oc_a\to \M$ is continuous and injective at every point of $\Oc_a$, but not necessarily has closed range. 
We then say that the inclusion map is a \textit{weak immersion} as e.g., in \cite{BGJP19}. 
If the range of $i_{*}$ at every point of $\Oc_a$ is closed, we say that $i$ is an \textit{immersion}. 
(By Definition \ref{defsubma}, this is a necessary but not sufficient condition for the inclusion to be an embedding).
\end{remark}

\begin{definition}\label{vcero}
	Let $\M$ be a $W^*$-algebra and $a\in\M^\nor$ with its corresponding action map $\alpha_a:\Vc_a\to \Oc_a$, $\alpha_a(v)=vav^*$. 
	For fixed $v_0\in \Vc_a$, we denote 
$$\delta_0=(\alpha_a)_{*v_0}:T_{v_0}\Vc_a\to \M,\quad \delta_0(x)=xa v_0^*+v_0ax^*.$$ 
Let $a_0:=v_0av_0^*\in \Oc_a\subseteq\M^\nor$, $p_0:=v_0v_0^*$, 
and $E_0 \colon \M\to\M_{a_0}:=\{x\in\M\mid xa_0=a_0x\}$ be a conditional expectation given by Lemma~\ref{BR05_Th2.2}. 
(It is noteworthy that the conditional expectation $E_0$ may be neither  weakly continuous nor unique.)
\end{definition}

We now show that for operators of finite spectrum, Example \ref{grass} is generic in a certain sense:

\begin{theorem}\label{immersed}
Let $a\in \M^{\nor}$. 
If $\spec(a)$ is a finite set, then $\Oc_a$ is the disjoint union of the unitary orbits of $a_i=\nu_i a\nu_i^*$ for certain $\nu_i\in \Vc_a$; these unitary orbits $\Oc(a_i)$ are the connected components of $\Oc_a$, and they are open and closed in $\Oc_a$. 
In particular $\Oc_a\subseteq \M$ is a closed  embedded split submanifold.
\end{theorem}

\begin{proof}
If $\spec(a)$ is finite, the groupoid orbit $\Oc_a$ is closed by Corollary~\ref{C4}, and so are the unitary orbits of any $a_0\in \Oc_a$ (Theorem \ref{finisemb}); these unitary orbits $\Oc(a_0)$  are clearly disjoint and contained in $\Oc_a$. 

We claim that there exists $\varepsilon>0$ depending only on $\spec(a)$ such that for any $a_0\in \Oc_a$
$$
\Oc_a\cap \{x:\|x-a_0\|<\varepsilon\}=\Oc(a_0)\cap \{x:\|x-a_0\|<\varepsilon\}.
$$
This shows that the unitary orbits are open in $\Oc_a$, and since they are embedded split submanifolds of $\M$ (Theorem \ref{finisemb}), the assertion on the differentiable structure of~$\Oc_a$ follows. 
To prove the claim, write $a=\sum\limits_{j=1}^n \lambda_j q_j$,  let $a_0:=v_0av_0^*\in\Oc_a$ for fixed $v_0\in \Vc_a$, and let let $b=v av^*\in \Oc_a$ with $v\in\Vc_a$ and $\|b-v_0 av_0^*\|<\delta$. Then 
$$
a_0=\sum_{j=1}^n \lambda_j v_0 q_j v_0^*=\sum_{j=1}^n \lambda_j P_j,\qquad b=\sum_{k=1}^n \lambda_j v q_jv^*=\sum_{j=1}^n \lambda_j Q_j, 
$$
where $P_j:=v_0 q_j v_0^*$ and $Q_j:=v q_jv^*$ for $j=1,\dots,n$. For each $j=1,\dots,n$, let $f_j:\mathbb C\to\mathbb R$
be the polynomial of degree $n-1$ 
with $f_j(\lambda_i)=0$ if $i\ne j$, and $f_j(\lambda_j)=1$, say $f_j(z)=\sum\limits_{k=1}^{n-1}\alpha_k z^k$. 
If $g_j(z):=\sum\limits_{k=2}^{n-1}k|\alpha_k|z^{k-1}$, then by \cite[Corollary 4]{dragomir} we have
$$
\|Q_j-P_j\|=\|f_j(b)-f_j(a_0)\|\le g_j(\|a\|)\|b-a_0\|<g_j(\|a\|)\delta.$$
Therefore if $\delta>0$ is sufficiently small, 
we have $\|v q_j v^*-v_0q_j v_0^*\|<\varepsilon$, and we can achieve 
$$
\|vv^*-v_0v_0^*\|=\|v p_av^*-v_0p_av_0^*\|=\|\sum_{j=1}^n v q_jv^*-v_0q_jv_0^*\|<1.$$
Hence, there exists $u\in \U(\M)$  such that $v=u v_0^*$ (in fact, $u$ depends smoothly on $v$, see  \cite[Section 2]{anco04} and the references therein). Thus we see that $b=ua_0u^*\in\Oc(a_0)$, and we are done.
\end{proof}

\begin{lemma}\label{time}
Let $a\in \Bc(\Hc)^{\nor}$. 
If $\delta_0, E_0$ are as in Definition \ref{vcero} for $\M=\Bc(\Hc)$, then 
the following assertions hold true. 
\begin{enumerate}[{\rm(1)}]
\item
\label{time_item1}
$T_{v_0}\Vc_a
=\{x\in\Bc(\Hc)\mid x^*v_0+v_0^*x=0
\}=\{zv_0\mid z^*=-z\in \Bc(p_0\Hc)\}$.
\item
\label{time_item2}
$\lie(K_{a_0})=\Ker\delta_0=v_0 \lie(K_a)v_0^*=\{zv_0\mid z^*=-z\in \Bc(p_0\Hc)\cap \Ran E_0\}$.
\item
\label{time_item3}
$T_{v_0}\Vc_a= 
\lie(K_{a_0})\oplus S_0$, 
where $S_0:=\{ wv_0\mid w^*=-w\in \Bc(p_0\Hc)\cap \Ker E_0\}$.
\item
\label{time_item4}
$\Ran\delta_0=\Ran(\delta_0|_{S_0})=\{[w,a_0] \mid w^*=-w \in \Bc(p_0\Hc)\cap \Ker E_0\}$.
\item
\label{time_item5}
$\overline{\Ran\delta_0}=\Bc(p_0\Hc)\cap \Ker E_0$.
\item
\label{time_item6}
$\Ran\delta_0$ is closed in $\Bc(\Hc)$ if and only if $\spec(a)$ is finite.
\end{enumerate}
\end{lemma}

\begin{proof} 
The first and second assertions can be proved with elementary algebraic manipulations using the definitions, and since $\id=(\id-E_0)+E_0$ so is the third one. 
In fact, these facts reflect a natural principal connection  carried by the principal $K_a$-bundle $\alpha_a\colon\Vc_a\to\Oc_a$. 
See \cite[Eq. (4.33)]{BO19} and the end of \cite[Subsect. 2.5]{BO19}. 

The fourth assertion follows from the previous one. 
In fact, if $x=wv_0\in S_0$ with $w^*=-w\in \Bc(p_0\Hc)\cap \Ker E_0$, 
then 
$$\delta_0(x)=xa v_0^*+v_0ax^*=wv_0a v_0^*+v_0av_0^*w^*=wa_0-a_0w.$$
From there everything happens inside $\Bc(p_0\Hc)$, and it is apparent apparent that $\Ran\delta_0 = \{[w,a_0]:w^*=-w \in \Bc(p_0\Hc)\cap \Ker E_0\}$.

If the spectrum of $a$ is finite, so is the spectrum of $a_0$ and with the same proof as in Lemma \ref{finite} (but in $\Hc_0=p_0\Hc$ instead of $\Hc$), we obtain that the range of $\delta_0$ is closed and in fact equal to $\{[a_0,z]: \, z^*=-z\in \Bc(p_0\Hc)\cap \Ker E_0\}$, thus we also have the proof of \eqref{time_item5} in this case. 
Now  assume that the range of $\delta_0$ is closed; first consider the case of $p_0$ of finite dimension, thus $a$ has finite range, and it must have finite spectrum.
 If the dimension of $p_0$ is infinite, so is the dimension of $\Hc_0=p_0\Hc$, and we are in the situation of  Theorem~\ref{au_th}, thus the spectrum of $a_0$ is finite, therefore the spectrum of $a$ is finite (and we can also we conclude that \eqref{time_item5} holds).
 \end{proof}

\begin{corollary}
	If $\Oc_a\hookrightarrow \Bc(\Hc)$ is immersed, then $\spec(a)$ is finite and $\Oc_a$ is in fact and embedded closed split submanifold.
\end{corollary}

\begin{proof}
If  the groupoid orbit $\Oc_a\hookrightarrow \Bc(\Hc)$ is immersed, then $\Ran\delta_0$ is closed in $\Bc(\Hc)$, hence Lemma~\ref{time}\eqref{time_item6} implies that  $\spec(a)$ is finite. 
Therefore, by Theorem~\ref{immersed}, $\Oc_a$ is and embedded closed split submanifold of~$\Bc(\Hc)$. 
\end{proof}

\subsection{Comparison of topologies}

As mentioned in Definition \ref{sometops}, the manifold topology  $\tau_q$ of $\Oc_a$ is usually finer (has more open sets) than the norm topology $\tau_{\M}$. In what follows we study the relation among the two topologies in more detail, for the case of $\M=\Bc(\Hc)$.


It is well-known that the if $u_n$ is the unitary shifting 
$e_1\mapsto e_2\mapsto \cdots \mapsto e_n\mapsto e_1$ for any orthonormal basis of $\Hc$, then $u_n$ converges strongly to the shift operator $e_k\mapsto e_{k+1}$. 
(Compare the proof of Theorem~\ref{isclosed}.)
With this example in mind, we can prove the following:

\begin{lemma}\label{finer}
Let $a\in \Bc(\Hc)^{\nor}$ such that $\spec(a)$ is infinite and accumulates only at $\lambda=0$. 
Then $\tau_q$ is strictly finer than $\tau_{\M}$ in $\Oc_a$.
\end{lemma}

\begin{proof} 
Since the non-zero spectral values of $a$ are isolated hence eigenvalues, we can write $a=\sum\limits_{k=1}^{\infty} \lambda_k Q_k$ with $Q_k$ disjoint orthogonal projections such that $\sum\limits_{k=1}^{\infty}Q_k=p_a$, and $\lambda_k\in\mathbb C$ the eigenvalues of $a$. 
We further assume that $|\lambda_k|\searrow 0$. 
For each $k\ge 1$, pick a unit vector $\xi_k\in \Ran Q_k$ and let $p_k=\xi_k\otimes \overline{\xi_k}$ be the orthogonal projection onto the 1-dimensional subspace~$\CC\xi_k$. 
Let $u$ be the shift of these vectors, i.e., 
$$
u=\sum_{k=1}^{\infty} \xi_{k+1}\otimes \overline{\xi_k}+ Q_k-p_k\in \Bc(\Hc).
$$
Then we compute $u^*u$ as follows:
\begin{align*}
u^*u
=\big(\sum_{k=1}^{\infty} \xi_k\otimes \overline{\xi_{k+1}}+ Q_k-p_k \big)
\big(\sum_{j=1}^{\infty} \xi_{j+1}\otimes \overline{\xi_j}+Q_j-p_j\big)
& = \sum_{k=1}^{\infty} p_k +Q_k-p_k \\
&=\sum_{k=1}^{\infty}Q_k =p_a,
\end{align*}
since $(Q_k-p_k)(\xi_j\otimes \overline{\xi_{j+1}})=0$ for each $j,k\ge 1$. Therefore $u\in \Vc_a$. 
On the  other hand, 
\begin{align*}
ua =
&\big(\sum_{k=1}^{\infty} \xi_{k+1}\otimes \overline{\xi_k}+ Q_k-p_k \big)
\big(\sum_{j=1}^{\infty} \lambda_j Q_j\big) 
=\sum_{k=1}^{\infty} \lambda_k\big(   \xi_{k+1}\otimes \overline{\xi_k} + Q_k-p_k \big),
\end{align*}
therefore
\begin{equation}
\label{finer_proof_eq1}
uau^*=\sum_{k=1}^{\infty} \lambda_k (Q_k-p_k)+\sum_{k=1}^{\infty} \lambda_k p_{k+1}.
\end{equation}
By taking $\lambda_k\equiv 1$, it is also clear that
\begin{equation}
\label{finer_proof_eq3}
uu^*=up_au^*=\sum_{k=1}^{\infty}  Q_k-p_k+ p_{k+1}=p_a-p_1.
\end{equation}
Now for each $n\ge 2$, let 
$$
u_n:=\sum_{k=1}^{n-1} \xi_{k+1}\otimes \overline{\xi_k}
+ \xi_1\otimes \overline{\xi_n} +\sum_{k=1}^n (Q_k-p_k)+\sum_{k\ge n+1}Q_k 
+(\1-p_a).
$$
Then $u_n\in\Bc(\Hc)$ is the unitary operator sending $\xi_1\mapsto \xi_2\ \mapsto \cdots \mapsto \xi_n\mapsto \xi_1$ and fixing the orthogonal complement of the range of $\sum\limits_{k=1}^n p_k$. 
It is then clear that for $k\ge n+1$, we have $u_n p_k=p_k$, that $u_np_nu_n^*=p_1$,  and $u_n p_ju_n^*=p_{j+1}$ for $1\le j \le n-1$. 
Thus  $u_n^*p_k u_n=p_k$ for $k\ge n+1$, and moreover
\begin{equation}\label{upun}
u_n^*p_au_n=p_a, \quad  u_n^* p_1 u_n=p_n,  \quad \textrm{ and } \quad u_n^*p_{j+1}u_n=p_j \textrm{ for  }1\le j\le n-1.
\end{equation}
We now compute
\begin{align*}
u_n a= &  
\big(\sum_{k=1}^{n-1} \xi_{k+1}\otimes \overline{\xi_k}
+ \xi_1\otimes \overline{\xi_n} 
+\sum_{k=1}^n (Q_k-p_k)
+\sum_{k\ge n+1}Q_k+(\1-p_a)\big)
\big(\sum_{j\ge 1}\lambda_j Q_j\big)\\
= & \sum_{k=1}^{n-1}\big( \lambda_k \, \xi_{k+1}\otimes \overline{\xi_k}\big)+  \lambda_n\, \xi_1\otimes \overline{\xi_n}+\sum_{k=1}^n \lambda_k (Q_k-p_k)
+\sum_{k\ge n+1}\lambda_k Q_k,
\end{align*}
and
\begin{equation}
\label{finer_proof_eq2}
u_nau_n^*=\lambda_n p_1+ \sum_{k=1}^{n-1} \lambda_k p_{k+1}+\sum_{k=1}^n \lambda_k(Q_k-p_k)+\sum_{k\ge n+1}\lambda_k Q_k.
\end{equation}
Let $b_n=u_n^*uau^*u_n \in \Oc_{uau^*}=\Oc_a$, since $u\in \Vc_a$. 
We claim that $\|b_n-a\|\to 0$ but that $b_n$ does not converge to $a$ in $\tau_q$. With this, it is apparent that the identity map $i:(\Oc,\tau_{\M})\to (\Oc_a,\tau_q)$ is not continuous, therefore the conclusion of the lemma follows. 

To prove the above claim, we first compute, using \eqref{finer_proof_eq1} and \eqref{finer_proof_eq2}, 
\begin{align*}
\|b_n-a\| & = \|uau^*-u_nau_n^*\| =\|\lambda_n p_1+ \sum_{k=n}^{+\infty}(\lambda_k-\lambda_{k+1})p_{k+1}+\lambda_{n+1}Q_{n+1}\|\\
&= \sup_{k\ge n+1} \{ |\lambda_n|, |\lambda_{k+1}-\lambda_k|,\vert\lambda_{n+1}\vert\}\le 2|\lambda_n|\to 0.
\end{align*}
Now let $W=\{w aw^*: w^*w=p_a,\, \|w-p_a\|<1\}$ be as in Definition \ref{sometops}. 
We claim that $b_n\not\in W$ for any $n\ge 1$, and this will finish the proof. 
If that is not the case, suppose that there exists $w\in \Vc_a$ such that $\|w-p_a\|<1$ and that 
$$
u_n^*uau^*u_n=b_n=waw^*.
$$
Then, for each $k\ge 1$,
\begin{align*}
wa^k w^*=(waw^*)^k=u_n^*(uau^*)^ku_n= u_n^* u a^k u^*u_n.
\end{align*}
Let $f$ be a continuous function such that $f\equiv 1$ in the first $j$ eigenvalues of $a$, and is zero in a neighborhood of $0$ and every other eigenvalue of $a$. Then $f$ can be uniformly approximated by polynomials with no constant term, therefore, denoting $P_j:=\sum\limits_{k=1}^jQ_k$, 
$$
w P_j w^*=w f(a)w^*=u_n^* u f(a)u^*u_n=u_n^*u P_j u^* u_n.
$$
The left-hand side converges in the strong operator topology to $w p_aw^*$, and the right hand side to $u_n^* u p_a u^* u_n=u_n^* uu^*u_n$. 
Therefore by \eqref{upun} and \eqref{finer_proof_eq3}, 
\begin{equation}\label{ww}
ww^*=wp_aw^*=u_n^*(p_a-p_1)u_n=p_a-p_n
\end{equation}
and $w^*=w^*ww^*=w^*(p_a-p_n)$. Hence $w=p_aw-p_nw$, thus $p_nw=p_nw-p_nw=0$. This proves that $w=p_aw$, and since $wp_a$ follows from $w^*w=p_a$, it is clear that $w\in \Bc(p_a\Hc)$. Now the condition $\|w-p_a\|<1$ implies that $w$ is invertible in $\Bc(p_a\Hc)=p_a\Bc(\Hc)p_a$, and since $w^*w=p_a$, it must be $ww^*=p_a$, a contradiction with (\ref{ww}).
\end{proof}

\begin{theorem}
\label{top_th}
Let $a\in \Bc(\Hc)^{\nor}$.  
If $\Oc_a$ is norm closed and $\tau_q=\tau_{\M}$, then  $\spec(a)$ is a finite set,  therefore $\Oc_a\subseteq \Bc(\Hc)$ is a closed  embedded split submanifold.
\end{theorem}

\begin{proof}
If the groupoid orbit is closed, by Corollary \ref{C4},  $\spec(a)$ accumulates at most at $\lambda=0$. 
If $\spec(a)$ is infinite, then we are in the situation of 
Lemma~\ref{finer}, 
therefore $\Oc_a$ is not embedded, a contradiction. 
Thus $\spec(a)$ must be a finite set. 
Everything else follows from  Theorem \ref{immersed}.
\end{proof}

Combining the previous results we can draw the following conclusion. 

\begin{corollary}
The closure of the unitary orbit $\Oc(a)$ of a compact normal operator $a\in \Bc(\Hc)$ with infinite spectrum (which is the groupoid orbit $\Oc_a$ of $a$) is a weakly immersed submanifold of $\Bc(\Hc)$ but never and immersed manifold, and its manifold topology is always finer than the operator norm topology.
\end{corollary}

We end this paper with the following problem:

\begin{problem}
Let $a=\sum\limits_{j\ge 1} \frac{1}{j}p_j + \sum\limits_{k\ge 1} (1+\frac{1}{k})q_k$ 
where $\{p_j\}_{j\ge 1}$ and $\{q_k\}_{k\ge 1}$ are sequences of mutually orthogonal projections in the $W^*$-algebra $\M=\Bc(\Hc)$. 
Assume further that the projections $P:=\sum\limits_{j\ge 1} p_j$ and $Q:=\sum\limits_{k\ge 1} q_k$ satisfy $P+Q=\1$. 
Then  the groupoid orbit $\Oc_a\subseteq \Bc(\Hc)$ is not an immersed manifold, and it is also not closed. 
Is it locally closed? Does $\tau_q=\tau_{\M}$ hold?
\end{problem}

\section*{Acknowledgment} 
The research of the first-named author (D.B.) was supported by a grant of the  Ministry of Research, Innovation and Digitization, CNCS/CCCDI -- UEFISCDI, project number PN-III-P4-ID-PCE-2020-0878, within PNCDI III. 
The research of the second-named author (G.L.) was supported by CONICET,  ANPCyT and Universidad de Buenos Aires (Argentina).



\begin{thebibliography}{1000000000}

\bibitem[AnCo04]{anco04}
E. Andruchow, G. Corach, 
{\it Differential geometry of partial isometries and partial unitaries.}
Illinois J. Math. {\bf 48} (2004), no. 1, 97--120.


\bibitem[AnCoMb05]{AnCoMb05} 
E.~Andruchow, G.~Corach, M.~Mbekhta, 
{\it On the geometry of generalized inverses}. 
Math. Nachr. {\bf 278} (2005), no. 7-8, 756--770.
	
\bibitem[AnLa10]
{AL10}
E.~Andruchow, G.~Larotonda,   
{\it The rectifiable distance in the unitary Fredholm group}. 
Studia Math. {\bf 196} (2010), no.~2, 151--178.

\bibitem[AnLaRe10]
{ALR10}
E.~Andruchow, G.~Larotonda, L.~Recht,   
{\it Finsler geometry and actions of the $p$-Schatten unitary groups}. 
Trans. Amer. Math. Soc. {\bf 362} (2010), no.~1, 319--344. 
	
\bibitem[AnSt89]
{AS89}
E.~Andruchow, D.~Stojanoff, 
{\it Differentiable structure of similarity orbits}. 
J. Operator Theory {\bf 21} (1989), no. 2, 349--366. 

\bibitem[Be06]
{B06}
D.~Belti\c t\u a, 
``Smooth homogeneous structures in operator theory''. 
Monographs and Surveys in Pure and Applied Mathematics, 137. Chapman \& Hall/CRC, Boca Raton, FL, 2006. 

\bibitem[BGJP19]{BGJP19}
D.~Belti\c t\u a, T.~Goli\'nski, G.~Jakimowicz, F.~Pelletier,  
{\it Banach-Lie groupoids and generalized inversion}. 
J. Funct. Anal. {\bf 276} (2019), no. 5, 1528--1574. 

\bibitem[BO19]{BO19}
D.~Belti\c t\u a,  A. Odzijewicz, 
{\it Poisson geometrical aspects of the Tomita-Takesaki modular theory}. 
Preprint arXiv: 1910.14466 [math.OA]. 



\bibitem[BePr07]{BePr07}
D.~Belti\c t\u a, B.~Prunaru, 
{\it Amenability, completely bounded projections, dynamical systems and smooth orbits}. 
Integral Equations Operator Theory {\bf 57} (2007), no. 1, 1--17. 

\bibitem[BeRa05]
{BR05}
D.~Belti\c t\u a, T.S.~Ratiu, 
{\it Symplectic leaves in real Banach Lie-Poisson spaces}. 
Geom. Funct. Anal. {\bf 15} (2005), no.~4, 753--779. 

\bibitem[BoVa16]{BoVa16}
T.~Bottazzi, A.~Varela, 
{\it Minimal length curves in unitary orbits of a Hermitian compact operator}. 
Differential Geom. Appl. {\bf 45} (2016), 1--22.

\bibitem[Bou06]{Bou06}
N.~Bourbaki, {\it Topologie g\'en\'erale}. Vh. 5 \`a 10, Springer, 2006. 

\bibitem[ChIo13]
{ChiDIyL13}
E.~Chiumiento, M.E.~Di Iorio y Lucero, 
{\it Geometry of unitary orbits of pinching operators}. 
J. Math. Anal. Appl. {\bf 402} (2013), no.~1, 103--118. 

\bibitem[DeFi79]{DeFi79}
D.~Deckard, L.A.~Fialkow:
{\it Characterization of Hilbert space operators with unitary cross sections}. 
J. Operator Theory {\bf 2} (1979), no.~2, 153--158. 

\bibitem[Di50]{Di50}
J.~Dixmier, 
{\it Les moyennes invariantes dans les semi-groupes et leurs applications}. 
Acta Sci. Math. (Szeged) {\bf 12} (1950), 213--227.

\bibitem[Doug66]{douglas} 
R.G.~Douglas, 
{\it On majorization, factorization, and range inclusion of operators on Hilbert space}. 
Proc. Amer. Math. Soc. {\bf 17} (1966), 413--415.

\bibitem[Dra16]{dragomir}
S. Dragomir, 
{\it Integral inequalities for Lipschitzian mappings between two Banach spaces and applications}.
Kodai Math. J. {\bf 39} (2016), no. 1, 227--251.

\bibitem[Fi75]
{Fi75}
L.A.~Fialkow,  
{\it The similarity orbit of a normal operator}. 
Trans. Amer. Math. Soc. {\bf 210} (1975), 129--137.


\bibitem[Fi77]
{Fi77}
L.A.~Fialkow,  
{\it A note on limits of unitarily equivalent operators}. 
Trans. Amer. Math. Soc. {\bf 232} (1977), 205--220.

\bibitem[Fi78]{Fi78}
L.A.~Fialkow,  
{\it A note on unitary cross sections for operators}. 
Canadian J. Math. {\bf 30} (1978), no. 6, 1215--1227. 

\bibitem[Fi79]
{Fi79}
L. A.~Fialkow, 
{\it A note on norm ideals and the operator $X \rightarrow  AX-XB$}. 
Israel J. Math. {\bf 32} (1979), no. 4, 331-348.

\bibitem[GoKr69]{GoKr69}
I.C.~Gohberg, M.G.~Kre\u\i n,
``Introduction to the Theory of Linear Nonselfadjoint Operators''.
Translations of Mathematical Monographs, Vol. 18,
American Mathematical Society, Providence, R.I., 1969.

\bibitem[GePa74]{GePa74}
R.~Gellar, L.~Page, 
{\it Limits of unitarily equivalent normal operators}. 
Duke Math. J. {\bf 41} (1974), 319--322. 

\bibitem[GKMS18]{GKMSu17}
J.~Grabowski, M.~Kus, G.~Marmo, T.~Shulman,  
{\it Geometry of quantum dynamics in infinite-dimensional Hilbert space}. 
J. Phys. A {\bf 51} (2018), no.~16, 165301.

\bibitem[Ha82]{Ha82}
P.R.~Halmos, 
``A Hilbert space problem book''. Second edition.
Graduate Texts in Mathematics,
Springer-Verlag, New York-Berlin, 1982. 

\bibitem[HK77]{HK77}
L.A.~Harris, W.~Kaup, 
{\it Linear algebraic groups in infinite dimensions}. 
Illinois J. Math. 21 (1977), no. 3, 666--674.

\bibitem[KaWe11]
{kw11} 
G.~Kaftal, G.~Weiss, 
{\it Majorization and arithmetic mean ideals}.
Indiana Univ. Math. J. {\bf 60} (2011), no. 5, 1393--1424. 

\bibitem[La19]{La19}
G.~Larotonda, 
\textit{Estructuras geom\'etricas para las variedades de Banach}. 
Book preprint, 2019. 
http://mate.dm.uba.ar/$\sim$glaroton/notas\_gl.html


\bibitem[Ne04]{Ne04}
K.-H.~Neeb, 
{\it Infinite-dimensional groups and their representations}. 
In: J.-Ph. Anker, B. \O rsted (eds.), \textit{Lie theory}, Progr. Math., 228, Birkhäuser Boston, Boston, MA, 2004, pp.~213--328. 

\bibitem[OdSl16]{OS} 
A. Odzijewicz, A. Sli\.{z}ewska, 
{\it Banach-Lie groupoids associated to $W^*$-algebras}.  
J. Symplectic Geom. {\bf 14} (2016), no. 3, 687--736. 

\bibitem[Ra77]
{rae77}
I.~Raeburn, 
{\it The relationship between a commutative Banach algebra and its maximal ideal space}. 
J. Functional Analysis {\bf 25} (1977), no. 4, 366--390. 

\bibitem[Sch60]
{Sch60}
R.~Schatten, 
``Norm ideals of completely continuous operators''. 
Ergebnisse der Mathematik und ihrer Grenzgebiete. N. F., Heft 27 Springer-Verlag, Berlin-G\"ottingen-Heidelberg, 1960. 

\bibitem[Sh07]
{Sh07}
D.~Sherman,  
{\it Unitary orbits of normal operators in von Neumann algebras}. 
J. Reine Angew. Math. {\bf 605} (2007), 95--132.

\bibitem[Vg89]
{varga}
J.V.~Varga, 
{\it Traces on irregular ideals}.
Proc. Amer. Math. Soc. {\bf 107} (1989), no. 3, 715--723.

\bibitem[Vo76]{Vo76}
D. Voiculescu, 
{\it A non-commutative Weyl-von Neumann theorem}. 
Rev. Roumaine Math. Pures Appl. {\bf 21} (1976), no. 1, 97--113.

\end{thebibliography}
\end{document}